\documentclass[10pt]{amsart}
\usepackage{cases}
\usepackage[utf8]{inputenc}
\usepackage{amsmath,amsthm,amssymb,amscd}
\usepackage{stmaryrd}
\usepackage[all,cmtip]{xy}
\usepackage{mathrsfs}
\usepackage{color}
\usepackage{amsmath}
\usepackage{amsxtra}
\usepackage{amscd}
\usepackage{amsthm}
\usepackage{amsfonts}
\usepackage{amssymb}
\usepackage{tikz-cd}
\usepackage{bbm}
\usepackage{hyperref}
\usepackage[mathscr]{eucal}
\usepackage[english]{babel}
\usepackage[autostyle]{csquotes}
\usepackage{mathtools}

\numberwithin{equation}{section}
\newtheorem{theorem}[equation]{Theorem}
\newtheorem{thm}{Theorem}
\theoremstyle{plain}
\newtheorem{lemma}[equation]{Lemma}
\newtheorem{proposition}[equation]{Proposition}
\newtheorem{definition}[equation]{Definition}
\newtheorem{corollary}[equation]{Corollary}

\newtheorem*{corollary*}{Corollary}

\newtheorem{remark}{Remark}

\newenvironment{myproof}[2] {\emph{Proof of {#1} {#2}.}}{\hfill$\square$}

\def\GL{\mathrm{GL}}
\def\SL{\mathrm{SL}}

\def\GSpin{\mathrm{GSpin}}
\def\SU{\mathrm{SU}}
\def\Spin{\mathrm{Spin}}

\def\GU{\mathrm{GU}}

\def\Nilp{(\mathrm{Nilp})}
\def\det{\mathrm{det}}
\def\Lie{\mathrm{Lie}}
\def\inv{\mathrm{inv}}

\DeclareMathOperator{\End}{End}
\DeclareMathOperator{\chara}{char}
\DeclareMathOperator{\Adm}{Adm}

\def\calM{\mathcal{M}}
\def\calN{\mathcal{N}}
\def\calO{\mathcal{O}}

\def\gothS{\mathfrak{S}}

\def\AAA{\mathbb{A}}

\def\CC{\mathbb{C}}
\def\DD{\mathbb{D}}
\def\FF{\mathbb{F}}
\def\GG{\mathbb{G}}

\def\PP{\mathbb{P}}
\def\QQ{\mathbb{Q}}

\def\XX{\mathbb{X}}
\def\ZZ{\mathbb{Z}}

\def\Sh{\gothS h}
\def\Shim{\mathrm{Sh}}
\def\Dieu{Dieudonn\'{e} module}

\newcommand{\length}{\mathrm{length}}
\newcommand{\vol}{\mathrm{vol}}

\newcommand{\Spf}[1]{\mathrm{Spf} (#1)}

\usepackage[top=1.5in, bottom=1.5in, left=1.2in, right=1.2in]{geometry}
\address{\parbox{\linewidth} {Haining Wang,\\ Department of Mathematics,\\ McGill University,\\ 805 Sherbrooke St W,\\ Montreal, QC H3A 0B9, Canada.~ }}
\email{wanghaining1121@outlook.com}

\date{\today}
\begin{document}

\title[On GU(2,2) type Rapoport-Zink spaces]{On the Bruhat-Tits stratification for GU(2,2) type Rapoport-Zink space: unramified case}

\author{Haining Wang}
\maketitle 
\begin{abstract}
In this note we study the supersingular locus of the $\GU(2,2)$ Shimura variety modulo a prime which is unramified in the imaginary quadratic extension. The supersingular locus of this Shimura variety can be related to the basic Rapoport-Zink space whose special fibre is described by the Bruhat-Tits stratification. The description for this supersingular locus in the case where the prime is inert in imaginary quadratic field is already known to Howard and Pappas by exploiting the exceptional isomorphism. Our method is more direct without using the exceptional isomorphism.
\end{abstract}
\tableofcontents

\section{Introduction}
\subsection{Motivation} In this note we study the supersingular locus of the $\GU(2,2)$ type Shimura variety at an odd prime $p$ which is unramified in the imaginary quadratic field. When $p$ is odd and inert in the corresponding imaginary quadratic field, the results in this note are already known to Howard and Pappas \cite{HP14}. Their method exploits the exceptional isomorphism between $\SU(2,2)$ and $\Spin(4,2)$. In this note, we show that one can work directly with the unitary group. For the $\GU(2,2)$ type Shimura variety, the supersingular locus agrees with the so-called basic locus for the Newton stratification and thus the problem of describing the supersingular locus is equivalent to the problem of describing the underlying reduced scheme of the basic Rapoport-Zink space according to the Rapoport-Zink uniformizaiton theorem. Our description is via the so-called \emph{Bruhat-Tits stratification}. This terminology comes from the pioneering work of \cite{Vol-can10} and \cite{VW-invent11} where the supersingular locus of the $\GU(1, n-1)$ type Shimura variety is studied. The Bruhat-Tits stratification roughly speaking comes from classifying the relative position between the lattices coming from the Dieudonn\'{e} modules of the $p$-divisible groups parametrized by the Rapoport-Zink space and the lattices representing the faces of the Bruhat-Tits building of the automorphism group of the isocrystal with additional structures associated to the base point of the Rapoport-Zink space. Our direct approach has certain advantage over the approach via the exceptional isomorphism: one can expect to use the results in this article to describe the $\GU(2,2)$ type Rapoport-Zink spaces with parahoric level structures using a similar approach as in \cite{Wang-2}. On the other hand, the $\GSpin$ type Rapoport-Zink space used by Howard and Pappas is of Hodge type and the construction of the Rapoport-Zink space with parahoric level structure is not constructed yet.  Our approach is also modelled over the group theoretic approach of \cite{GH-Cam15} and can be seen as a translation of their result from the equi-characteristic case to the present mixed characteristic case. 

\subsection{Integral model of Shimura variety}\label{integral-model} Let $E$ be an imaginary quadratic extension of $\QQ$ with ring of integers $\mathcal{O}_{E}$ and let $*$ be the non-trivial automorphism of $E$ over $\QQ$. Let $V$ be a Hermitian space with Hermitian form $(\cdot, \cdot)$ over $E$ of signature $(2,2)$. We denote by $G=\GU(V)$ the group of unitary similitudes of $V$.  We consider the Hodge cocharacter $h: \GG_{m}\rightarrow G(\CC)$ sending $z\in \CC^{\times}$ to $\text{diag}(z, z, 1, 1)$.

We assume there is an $\calO_{E}$ lattice $\Lambda\subset V$ which is selfdual with respect to $(\cdot,\cdot )$. 
Note $h$ defines a decomposition of $V_{\CC}=V_{1}\oplus V_{2}$ where $V_{1}$ is the subspace of $V$ that $h(z)$ acts  by $z$ and $V_{2}$ is the subspace of $V$ that $h(z)$ acts  by $\bar{z}$. We fix an open compact subgroup $U^{p}$ of $G(\AAA^{p}_{f})$ that we assume is sufficiently small and we define $U_{p}= G(\ZZ_{p})$. Finally we let $U=U_{p}U^{p}$. 

Given the PEL datum $(E, *, V, (\cdot,\cdot), h, \Lambda)$ cf.\cite{Ko92}, we consider the following moduli problem $\Sh_{U^{p}}$ over $\ZZ_{(p)}$. For a scheme $S$ over $\ZZ_{(p)}$,   $\gothS h_{U^{p}}(S)$ classifies the set of isomorphism classes of the quadruple $(A, \iota , \lambda, \eta )$ where
\begin{itemize}
\item[-]  $A$ is an abelian scheme of relative dimension $4$ over $S$;
\item[-]  $\lambda: A\rightarrow A^{\vee}$ is a polarization which is $\mathcal{O}_{E}$-linear and of degree prime to $p$;
\item[-] $\iota: \calO_{E} \rightarrow \End_{S}(A)$ is a morphism such that $\lambda\circ \iota(a^{*})\circ\lambda^{-1}= \iota(a)^{\vee}$ and satisfies the Kottwitz condition \begin{equation}\label{Kowtt}\det(\iota(a); \Lie(A))=\det(a; V_{1})\end{equation} for all $a\in \calO_{E}$. Note this equality has to be interpreted as the way explained in \cite[page 390]{Ko92};
\item[-] $\eta: V\otimes_{\QQ} {\AAA}^{(p)}_{f} \rightarrow V^{(p)}(A)$ is a $U^{p}$-orbit of $E \otimes_{\QQ}\AAA^{(p)}_{f}$-linear isomorphisms which respects the Weil-pairing on the righthand side and the form $(\cdot,\cdot)$ on the lefthand side. Here $V^{(p)}(A)=\prod_{p^{\prime}\neq p}T_{p^{\prime}}(A)\otimes \AAA^{(p)}_{f}$ is the prime to $p$-part of the rational Tate module of $A$. 
\end{itemize}

Let $\FF$ be an algebraically closed field containing $\FF_{p}$. In this note we are concerned with the structure of  the supersingular locus $\Shim^{ss}_{U^{p}}$ in $\Sh_{U^{p}}\otimes \FF$ considered as a reduced closed subscheme. The structure depends on whether $p$ is split or $p$ is inert in $E$. Indeed, our starting point of analyzing the supersingular locus is the uniformization theorem of Rapoport-Zink \cite[ Theorem 6.1]{RZ-Aoms} which takes the following form
\begin{equation}\label{RZ-uniform}
\alpha: I(\QQ)\backslash \calN_{\text{red}}\times G(\AAA^{p}_{f})/ U^{p} \xrightarrow{\sim  }  \Shim^{ss}_{U^{p}}. 
\end{equation}
Here $ \calN_{\text{red}}$ is the underlying reduced scheme of a Rapoport-Zink space $\calN$ and $I$ is an inner form of $G$. When $p$ is split in $E$, then the Rapoport-Zink space is of EL type and when $p$ is inert in $E$, then the Rapoport-Zink space is of PEL type. We will discuss the two cases separately in the next two subsections. First, we fix a base point $(\bf{A},\boldsymbol{\iota}, \boldsymbol{\lambda},\boldsymbol{ \eta})$ over $\FF$ in $\Shim^{ss}_{U^{p}}$. Let $E_{p}=E\otimes \QQ_{p}$.  We pass $\bf{A}$ to its associated $p$-divisible group $(\XX, \iota, \lambda)$ with  $\iota: E_{p}\rightarrow \End(\XX)_{\QQ}$ and $\lambda: \XX\rightarrow \XX^{\vee}$ induced by the above PEL structure on $\bf{A}$. Let $W_{0}=W(\FF)$ be the Witt ring of $\FF$ and let $K_{0}=W(\FF)_{\QQ}$ be the fraction field of $W_{0}$. 

\subsection{EL type}\label{EL-case}  In the case $p$ is split, $E_{p}=\QQ_{p}\times \QQ_{p}$ and $\XX=\XX_{0}\times \XX_{1}$. Let $\Nilp$ be the category of $W_{0}$-schemes S, with $p$ locally nilpotent on $S$. Then we are led to consider the Rapoport-Zink space $\calN$ which is the set valued functor that associates a scheme $S\in\Nilp$ the set of isomorphism classes of the pair $(X,  \rho_{X})$ where 
\begin{itemize}
\item[-] $X$ is a two dimensional $p$-divisible group over $S$;
\item[-] $\rho_{X}: X\times \bar{S}\rightarrow \XX_{0}\times \bar{S} $ is a quasi-isogeny over $\bar{S}$ which is the locus of $S$ where $p=0$.
\end{itemize}

This functor is representable by a smooth formal scheme $\calN$ over $\Spf{W_{0}}$. The formal scheme $\calN$ decomposes into connected components $\calN=\bigsqcup_{i\in \ZZ}\calN(i)$ where $\calN(i)$ is the sub formal scheme given by the condition that $\rho_{X}$ has height $i$. Then the underlying reduced scheme $\calN_{red}(0)$ of $\calN(0)$ can be described by the so called \emph{Bruhat-Tits stratification}. For the $p$-divisible group $\XX_{0}$, we consider its automorphism group $J_{b}$ whose derived group is isomorphic to $\SL_{B_{p}}(2)$ for the quaternion division algebra $B_{p}$ over $\QQ_{p}$. The Bruhat-Tits building of this group over $\QQ_{p}$ is a tree with two kinds of vertices correspond to the nodes $\{1, 3\}$ and $\{0, 2\}$ in the affine Dynkin diagram of type $\widetilde{A}_{3}$. They correspond to two different kinds of strata in the Burhat-Tits stratification. Our main result about this Rapoport-Zink space is the following theorem which one can find a more precise version in Theorem \ref{main-p-split}.

\begin{thm}\label{main-p-split-intro}
 Let $\calM=\calN_{red}(0)$. We have a stratification $$\calM=\calM^{\circ}_{\{1,3\}}\sqcup \calM_{\{0,2\}}$$ called the Bruhat-Tits stratification. The subscheme $\calM^{\circ}_{\{1,3\}}$ is open and dense in $\calM$ and $\calM_{\{0,2\}}$ is the set of superspecial points. The irreducible components of $\calM$ are isomorphic to $\PP^{1}$.
\end{thm}

\subsection{PEL type}\label{PEL-case} In the case $p$ is inert, then $E_{p}=\QQ_{p^{2}}$ and we denote by $\psi_{0}, \psi_{1}$ the two embeddings of $\QQ_{p^{2}}$ in $K_{0}$. We consider the following set valued functor $\calN$ on $\Nilp$. To a scheme $S\in\Nilp$, $\calN$ classifies the set of isomorphism classes of the quadruple $(X,\lambda_{X}, \iota_{X},  \rho_{X} )$ where 
\begin{itemize}
\item[-] $X$ is a two dimensional $p$-divisible group over $S$;
\item[-] $\lambda_{X}: X\rightarrow X^{\vee}$ is a prime to $p$ polarization over $S$;
\item[-] $\iota_{X}: \ZZ_{p^{2}}\rightarrow \End_{S}(X)$ satisfies $$\det(T-\iota(a); \Lie(X))= (T-\psi_{0}(a))^{2}(T-\psi_{1}(a))^{2}$$ in $\calO_{S}[T]$;
\item[-] $\rho_{X}: X\times \bar{S}\rightarrow \XX\times \bar{S} $ is a quasi-isogeny where $\bar{S}$ is the locus of $S$ where $p=0$.
\end{itemize}

This moduli space is representable by a smooth formal scheme $\calN$. The formal scheme $\calN$ can be decomposed into $\calN=\bigsqcup_{i\in\ZZ}\calN(i)$ where the open and closed sub formal scheme $\calN(i)$ is isomorphic to $\calN(0)$. Denote by $\calM=\calN_{red}(0)$ the underlying reduced scheme of $\calN(0)$. Then our main result concerns the structure of $\calM$. We again consider the group $J_{b}$ of automorphisms of $\XX$ that preserve the additional structures.  This group is isomorphic to a different unitary group of degree $4$.  The Bruhat-Tits building of the derived group of $J_{b}$ is obtained as the fixed point of the Bruhat-Tits building of the group $\SL(4)$ over $K_{0}$ under the Frobenius action. This action on the affine Dynkin diagram of type $\widetilde{A}_{3}$ fixes the nodes $1,3$ and exchanges the nodes $0$ and $2$. If we label the vertices of the base alcove of the building of $\SL(4)$ over $K_{0}$ according to the nodes on the affine Dynkin diagram of type $\widetilde{A}_{3}$, then the strata of the Bruhat-Tits stratification correspond to the vertices of the building of $J_{b}$ over $\QQ_{p}$ attached to the nodes $1$, $3$ and $\{0,2\}$ as well as the edge corresponding to $\{1,3\}$. The main result is the following and one can find more details in Theorem \ref{main-result-inert}.

\begin{thm}\label{main-p-inert-intro}
The scheme $\calM=\calN_{red}(0)$ is pure of dimension $2$ and admits the Bruhat-Tits stratification $$\calM=\calM^{\circ}_{\{1\}}\sqcup \calM^{\circ}_{\{3\}} \sqcup \calM^{\circ}_{\{1,3\}}\sqcup \calM_{\{2\}}.$$ 
The irreducible components of the closures of $\calM^{\circ}_{\{1\}}$ and $\calM^{\circ}_{\{3\}}$ are isomorphic to the Fermat hypersurface $x^{p+1}_{0}+x^{p+1}_{1}+x^{p+1}_{2}+x^{p+1}_{3}=0$.  The irreducible components of the closure of $\calM^{\circ}_{\{1,3\}}$ are isomorphic to $\PP^{1}$ and $\calM_{\{2\}}$ consists of superspecial points.
\end{thm}

\subsection{Acknowlegement}The author would like to thank Henri Darmon and Pengfei Guan for supporting his postdoctoral studies and staffs from McGill university for their help. We would like to thank the referee for helpful comments and suggestions to improve the expositions.

\subsection{Notations}Let $p$ be an odd prime and let $\FF$ be an algebraically closed field containing $\FF_{p}$. Let $W_{0}=W(\FF)$ be the Witt ring of $\FF$. Let $\sigma$ be the Frobenius on $\FF$. We denote by $\QQ_{p^{2}}$ the unramifeid quadratic extension of $\QQ_{p}$ and by $\ZZ_{p^{2}}$ its ring of integers. If $M_{1}\subset M_{2}$ are two $W_{0}$-modules, we write $M_{1}\subset^{d} M_{2}$ if the colength of the inclusion is $d$. Let $B_{p}$ be the quaternion division algebra over $\QQ_{p}$ and let $\calO_{B_{p}}$ be its maximal order. Let $E$ be an imaginary quadratic extension of $\QQ$ and let $\calO_{E}$ be the ring of integers of $E$. We denote by $*$ the nontrivial automorphism of $E$ over $\QQ$. We assume that $p$ is unramified in $E$ in this note. If $R$ is ring and $L$ is an $R$-module and $R^{\prime}$ is an $R$-algebra, we use the notation $L_{R^{\prime}}=L\otimes_{R} R^{\prime}$.  Let $G$ be a reductive group over $\QQ_{p}$, we denote by $B(G)$ the set of $\sigma$-conjugacy classes in $G(K_{0})$ following Kottwitz \cite{Ko-Comp85}.

\section{ Rapoport-Zink space of EL type}
\subsection{Rapoport-Zink space in the $p$ split case} In the case $p$ is split, $E_{p}=\QQ_{p}\times \QQ_{p}$, then $\bf{\iota}$ induces a decomposition $\XX=\XX_{0}\times \XX_{1}$.  The polarization $\lambda_{\XX}$ and the determinant condition identifies the dual of first factor as the second factor. Therefore we are led to consider the functor $\calN$ defined in Section \ref{EL-case}. Recall that for $S\in\Nilp$, $\calN(S)$ is the set of isomorphism classes of $(X,  \rho_{X})$ where 
\begin{itemize}
\item[-] $X$ is a two dimensional $p$-divisible group over $S$;
\item[-] $\rho_{X}: X\times \bar{S}\rightarrow \XX_{0}\times \bar{S} $ is a quasi-isogeny over $\bar{S}$ where $\bar{S}$ is the locus of $S$ where $p=0$.
\end{itemize}

Let $M_{0}$ be the {\Dieu} of $\XX_{0}$ and $N_{0}$ be its isoclinic isocrystal of slope $\frac{1}{2}$. Let $G=\GL(4)$ and we denote by $b\in B(G)$ the $\sigma$-conjugacy class associated to $N_{0}$.  We would like to have a concrete description of the group of automorphisms of the isocrystal $N_{0}$. This is the group defined by
\begin{equation}\label{J-split}
 J_{b}(\QQ_{p})=\{g\in G(K_{0}): g^{-1}b\sigma(g)=b\}.
\end{equation} 
We set $\tau= V^{-1}F$ where $F$ is the Frobenius operator on $N_{0}$ and $V^{-1}$ is the operator $p^{-1}F$. Then we consider the slope zero isocrystal $(N_{0}, \tau)$ and let $C_{0}$ be the $\QQ_{p^{2}}$ vector space $N^{\tau=1}_{0}$ of dimension 4. Notice that we can write $B_{p}$ as $\QQ_{p^{2}}+\QQ_{p^{2}}\Pi$ with $\Pi^{2}=p$ and by letting $\Pi$ act on $C_{0}$ by $V$ we can equip $C_{0}$ with a structure of a $B_{p}$-module. 

\begin{lemma}
The group $J_{b}(\QQ_{p})$ can be identified with $\GL_{B_{p}}(C_{0})$.
\end{lemma}
\begin{proof}
By definition \eqref{J-split},  $J_{b}(\QQ_{p})= \{g\in \End(N_{0})^{\times}_{\QQ}: Fg=gF\}$. Since $g$ commutes with $F$, it commutes with $V$ and $\tau$. Hence it commutes with the $B_{p}$-linear action on $C_{0}$. Then the conclusion is clear.
\end{proof}

\subsection{Set structure in the $p$ split case} Denote by $\calN(i)$ the open and closed sub formal scheme of $\calN$ defined by the condition that the quasi-isogeny $\rho$ has height $i$. Then by \cite[Proposition 1.18]{Vol-can10}, we have an isomorphism $\calN(0)\cong \calN(i)$ and we denote by $\calM=\calN_{red}(0)$ the underlying reduced scheme of $\calN(0)$. Let $M\subset N_{0}$ be a {\Dieu}, we define $\vol(M)=\length (M/M\cap M_{0})-\length (M_{0}/M\cap M_{0})$. 
\begin{lemma}
The set of $\FF$-points of $\calM$ can be identified with 
\begin{equation}\label{set-split}
\calM(\FF)=\{M\subset N_{0}: pM\subset^{2} VM\subset^{2} M, \vol(M)=0 \}.
\end{equation}

\end{lemma}
\begin{proof}
This is clear. The colength condition follows from the fact that the $p$-divisible groups are two dimensional.
\end{proof}

We are going to analyze the set $\calM(\FF)$ by decomposing it into pieces $\calM_{L}(\FF)$ indexed by the so-called \emph{vertex lattices} $L$ in $C_{0}$.

\begin{definition}
A vertex lattice $L$ in $C_{0}$ is a $\ZZ_{p^{2}}$-lattice in $C_{0}$ with the property that $\Pi L\subset L$ or in other words a $\calO_{B_{p}}$-lattice in $C_{0}$.
\end{definition}

\begin{remark}
Recall that the derived group $\tilde{J}$ of $J_{b}$ is $\SL_{B_{p}}(C_{0})$ and the Bruhat-Tits building of $\SL_{B_{p}}(C_{0})$ is a tree \cite[Example 2.7]{Tits-local}. The vertex lattices correspond to vertices of this tree.
\end{remark}

The next lemma shows that vertex lattices correspond to superspecial Dieudonn\'{e} modules.

\begin{lemma}\label{vertex}
Let $L$ be a vertex lattice, then $L_{W_{0}}$ is a Dieudonn\'{e} submodule in $N_{0}$ and is superspecial.
\end{lemma}
\begin{proof}
Since $L$ is $\tau$-invariant, it follows that $L_{W_{0}}$ is $\tau$-stable and thus $\Pi L_{W_{0}}=V L_{W_{0}}=F L_{W_{0}}$. Then it is clear that $pL_{W_{0}}\subset VL_{W_{0}}\subset L_{W_{0}}$. It also follows that $F^{2}L_{W_{0}}=p L_{W_{0}}$ and $L_{W_{0}}$ is superspecial.
\end{proof}

For each vertex lattice $L$, we define the corresponding \emph{lattice stratum} to be 
\begin{equation}\label{set-description}
\calM_{L}(\FF)=\{M\in\calM(\FF): \Pi L_{W_{0}}\subset M\subset L_{W_{0}}\}. 
\end{equation}
The following proposition guarantees that $\calM(\FF)=\bigcup_{L}\calM_{L}(\FF)$ where $L$ runs through all the vertex lattices.

\begin{proposition}\label{crucial_lemma_split}
For each $M\in \calM$, we can find a $\tau$-stable lattice $L\subset C_{0}$ such that $\Pi{L_{W_{0}}}\subset M\subset L_{W_{0}}$. 
\end{proposition}

Before we prove this proposition, we need to recall a few basic facts about Dieudonn\'{e} modules that are stated in \cite[(11a)-(11d)]{NO-Ann80}. We define the $a$-number of the Dieudonn\'{e} module $M$ by  
\begin{equation}
a(M)=\dim_{\FF} M/ (FM+VM).
\end{equation}
Notice that $a(M)=2$ if and only if $M$ is superspecial. 
\begin{lemma}\label{NO_lemma}
Suppose $M$ has $a(M)=1$, then we have
\begin{enumerate}
\item  $F^{2}(M)/pM$ is the unique $\FF[F,V]$-submodule of $F(M)/pM$ of rank $1$;
\item $F^{2}(M)+pM=FM\cap VM= V^{2}(M)+pM $.
\end{enumerate}
\end{lemma}
\begin{proof}
This is precisely the statement of $11(b), 11(c)$ and $11(d)$ in \cite{NO-Ann80}. 

\end{proof}

\begin{myproof}{Proposition}{\ref{crucial_lemma_split}} 
Suppose $a(M)=2$, then $M$ is superspecial and therefore $M$ is $\tau$-stable and the statement is trivial.

Suppose $M$ is not $\tau$-stable then $a(M)=1$. Consider $L_{W_{0}}= M+\tau(M)$. Then $pL_{W_{0}}=pM+F^{2}M$ and $p\tau^{-1}(L_{W_{0}})=p(\tau^{-1}(M)+M)=V^{2}M+pM$. Hence $L_{W_{0}}=\tau(L_{W_{0}})$ by Proposition \ref{NO_lemma} (2). Since $\Pi$ act on $L_{W_{0}}$ by $V$, $\Pi L_{W_{0}}=FM+VM\subset M$.
\end{myproof}

Let $L$ be a vertex lattice, for each $M\in \calM_{L}(\FF)$, $M$ either fits in 
\begin{equation}\label{P1-chain1}
VM\subset^{1} \Pi L_{W_{0}} \subset^{1} M\subset^{1} L_{W_{0}}
\end{equation}
or it fits in 
\begin{equation}\label{P1-chain2}
VM\subset^{0} \Pi L_{W_{0}} \subset^{2} M\subset^{0} L_{W_{0}}.
\end{equation}

We define the Ekedahl-Oort strata of $\calM_{L}(\FF)$ by
\begin{equation}
\begin{split}
&\calM^{\circ}_{L}(\FF)=\calM_{L}(\FF)- \bigcup_{L^{\prime}\neq L}\calM_{L^{\prime}}(\FF)\\
&\calM_{L, \{0,2\}}(\FF)=\calM_{L}(\FF)\cap \bigcup_{L^{\prime}\neq L}\calM_{L^{\prime}}(\FF).\\
\end{split}
\end{equation}
Then we have the natural decomposition $$\calM_{L}(\FF)=\calM^{\circ}_{L}(\FF)\sqcup \calM_{L,\{0,2\}}(\FF)$$ which we refer to as the Ekedahl-Oort stratification of $\calM_{L}(\FF)$. 
Notice that for $M\in\calM_{L, \{0,2\}}(\FF)$, $M=L_{W_{0}}\cap L^{\prime}_{W_{0}}$ for another vertex lattice $L^{\prime}$. This is easy to see by considering the index among lattices. Therefore $M=L_{W_{0}}\cap L^{\prime}_{W_{0}}$ corresponds to a vertex lattice contained in $L$.
Let $V=L_{W_{0}}/\Pi L_{W_{0}}$ and it is a two dimensional vector space over $\FF$.  Let $k$ be any field extension of $\FF$.

\begin{proposition}\label{P1-set}
There is bijection between  $\calM_{L}(k)$ and $\PP(V)(k)$. The map is given by sending $M$ to $M/\Pi L_{W_{0}}$. Moreover the isomorphism respects the stratification on both sides in the sense that  $\calM_{L, \{0,2\}}(k)$ corresponds to the $\FF_{p^{2}}$-rational points of $\PP(V)(k)$ and $\calM^{\circ}_{L}(k)$ corresponds to the complement of $\FF_{p^{2}}$-rational points in $\PP(V)(k)$.
\end{proposition}
\begin{proof}
For $k=\FF$, the map is given by sending $M$ to $M/\Pi L_{W_{0}}$ and it is clearly a bijection by \eqref{P1-chain1} and \eqref{P1-chain2}. For the second claim, notice that giving a $\FF_{p^{2}}$-point is equivalent to giving a line $l_{0}$ in $L/\Pi L$. Denote by $L_{0}$ the preimage of $l_{0}$ under the natural reduction map. Then $L_{0}$ is a $\ZZ_{p^{2}}$ lattice that $\Pi L_{0}\subset \Pi L\subset L_{0}\subset L$ and hence is a vertex lattice.  For general $k$, one can use the same proof by replacing {\Dieu} with Zink's theory of windows of displays \cite{Zink-pro99}. 
\end{proof}

\subsection{Bruhat-Tits stratification in the $p$-split case} Let $L$ be a vertex lattice. Now we will equip $\calM_{L}(k)$ with a scheme theoretic structure.  By \ref{vertex}, $L_{W_{0}}$ is a superspecial Dieudonn\'{e} module contained in $N_{0}$ and we can associate a $p$-divisible group $\XX_{L}$ to it and a quasi-isogeny $\rho_{L}: \XX_{L}\rightarrow \XX_{0}$. Similarly we can define $\XX_{\Pi L}$ associated to $\Pi L$ and  a quasi-isogeny $\rho_{\Pi L}: \XX_{\Pi L}\rightarrow \XX_{0}$. We define  quasi-isogenies through the composite
\begin{equation}\rho_{X,L}: X\xrightarrow{\rho_{X}} \XX_{0}\xrightarrow{\rho^{-1}_{L}}\XX_{L}\end{equation}
and
\begin{equation}\rho_{\Pi L, X}: \XX_{\Pi L}\xrightarrow{\rho_{\Pi L}} \XX_{0}\xrightarrow{\rho^{-1}_{X}} X.\end{equation}

Then we consider the subfunctor $\calM_{L}$ of $\calM$ given by all $(X, \rho_{X})\in \calM$ that the quasi-isogenies $\rho_{X,L}$ and $\rho_{\Pi L, X}$ are true isogenies.

\begin{lemma}\label{projective-split}
The subfunctor $\calM_{L}$ is representable by a projective scheme whose $\FF$ points are given as in \eqref{set-split}.
\end{lemma}

\begin{proof}
The subfunctor is representable by a closed subscheme of $\calM$  by \cite[Proposition 2.9]{RZ-Aoms}. Moreover $\calM_{L}$ is bounded in the sense of \cite[2.30]{RZ-Aoms}. This can be seen from the fact that we have construced isogenies
$$\XX_{\Pi L, R} \xrightarrow{\rho_{\Pi L, X}} X  \xrightarrow{\rho_{L,X}}\XX_{L, R}.$$
Thus $\calM_{L}$ is closed subscheme of a projective scheme by \cite[Corollary 2.29]{RZ-Aoms}. Thus it is projective itself.\end{proof}

We define the Ekedahl-Oort strata of $\calM_{L}$ by  
\begin{equation}
\begin{split}
&\calM^{\circ}_{L}=\calM_{L}- \bigcup_{L^{\prime}\neq L}\calM_{L^{\prime}}\\
&\calM_{L, \{0,2\}}=\calM_{L}\cap \bigcup_{L^{\prime}\neq L}\calM_{L^{\prime}}\\
\end{split}
\end{equation}
And we call the natural decomposition $\calM_{L}=\calM^{\circ}_{L}\sqcup \calM_{L, \{0,2\}}$ the Ekedahl-Oort stratification of $\calM_{L}$. 
Now we define a morphism $f: \calM_{L}\rightarrow \PP(V_{L})$. Let $R$ be a reduced $\FF$-algebra of finite type. Given $(X, \rho_{X})\in \calM_{L}(R)$, we consider the Grothedieck-Messsing crytal $\DD(X)(R)$ cf. \cite{BBM82}. The isogenies $\rho_{L, X}$ and  $\rho_{\Pi L, X}$ give an inclusion $\DD(\XX_{\Pi L})(R)=\Pi L_{W_{0}}\otimes R\subset\DD(X)(R)\subset \DD(\XX_{L})(R)=L_{W_{0}}\otimes R$ and we define $f$ simply by sending $(X, \rho_{X})\in \calM_{L}(R)$ to $\DD(X)(R)/\Pi L_{W_{0}}\otimes R$ in $L_{W_{0}}\otimes R/ \Pi L_{W_{0}}\otimes R$.

\begin{proposition}
 The map $f: \calM_{L}\rightarrow \PP(V_{L})$ is an isomorphism.  Moreover the isomorphism respects the stratification on both sides in the sense that $\calM_{L,\{0,2\}}$ corresponds to the $\FF_{p^{2}}$-points on $\PP(V)$ and $\calM^{\circ}_{L}$ corresponds to the complement of $\FF_{p^{2}}$-points in $\PP(V)$.
\end{proposition}
\begin{proof}
By \ref{P1-set}, $f$ is a bijection on points. Moreover using \ref{projective-split}, we see that $f$ is birational quasi-finite and proper. Therefore $f$ is an isomorphism by Zariski's main theorem since the base is normal. The moreover part can be checked on $\FF$-points which is done in Proposition \ref{P1-set}.
\end{proof}

\subsection{The main result in the $p$-split case} We now summarize the results in the previous subsections. We define the Bruhat-Tits strata by 
\begin{equation}
\begin{split}
&\calM^{\circ}_{\{1,3\}}=\bigcup_{L}\calM^{\circ}_{L} \text{ where $L$ runs through all the vertex lattices};\\
&\calM_{\{0,2\}}= \bigcup_{(L,L^{\prime})}\calM_{L}\cap \calM_{L^{\prime}} \text{ where $(L ,L^{\prime})$ runs through all distinct pair of vertex lattices}.\\
\end{split}
\end{equation}
We call the resulting decomposition $$\calM=\calM^{\circ}_{\{1,3\}}\sqcup \calM_{\{0,2\}}$$ the \emph{Bruhat-Tits stratification} of $\calM$.

\begin{remark}\label{diag-split}
We would like to explain our choice of terminology for the vertex lattices. The affine Dynkin diagram of type $\tilde{A}_{3}$ is a cycle with $4$-vertices. 
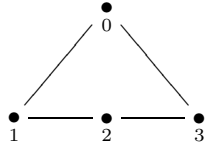
\begin{displaymath}
  \xymatrix{& &\underset{0}\bullet \ar@{-}[dl]\ar@{-}[dr]& \\
                 &\underset{1}\bullet \ar@{-}[r] &\underset{2}\bullet &\underset{3}\bullet\ar@{-}[l]}
\end{displaymath}
The derived group $\tilde{J}$  of $J_{b}$ is $\SL_{B_{p}}(C_{0})$ and the automorphism of the affine Dynkin diagram corresponds this group exchange the nodes $1$ and $3$ and exchanges the nodes $0$ and $2$.  See also \cite[Example 1.3]{Gor18} for the explanations of this.
\end{remark}

\begin{theorem}\label{main-p-split}
The reduced scheme $\calM$ is pure of dimension $1$. 
\begin{enumerate}
\item We have a stratification $\calM=\calM^{\circ}_{\{1,3\}}\sqcup \calM_{\{0,2\}}$ called the Bruhat-Tits stratification. The subscheme $\calM^{\circ}_{\{1,3\}}$ is open and dense in $\calM$ and $\calM_{\{0,2\}}$ is the set of superspecial points. Each irreducible component of $\calM$ is isomorphic to $\calM_{L}$ for a vertex lattice $L$.
\item The scheme $\calM_{L}$ is isomorphic to $\PP^{1}$ and the $\FF_{p^{2}}$-rational points of $\PP^{1}$ are precisely the superspecial points on $\calM_{L}$. 
\item For two vertex lattices $L$ and $L^{\prime}$, the intersection $\calM_{L}\cap\calM_{L^{\prime}}$ is a point which is superspecial.
\end{enumerate}
\end{theorem}

\begin{remark}
As the referee points out, in the above description the scheme $\calM_{L}$ is similar to the Moret-Bailly family which is used by Katsura and Oort \cite{KO-COM87} to study a similar global moduli problem. 
\end{remark}

 \subsection{Application to the supersingular locus: $p$-split case} Recall the integral model defined in section \ref{integral-model}. By the Rapoport-Zink uniformization theorem Theorem \ref{RZ-uniform}, we can transfer  the results we have obtained in the previous subsections to a description of the supersingular locus $\Shim^{ss}_{U^{p}}$.  This is the closed subscheme of $\Sh_{U^{p}, \FF}$ defined by
\begin{equation}
\Shim^{ss}_{U^{p}}=\{(A, \lambda, \iota, \eta)\in \Sh_{U^{p},\FF}: \text{ $A$ is supersingular}\}.
\end{equation}
Here we only consider its underlying reduced scheme.

\begin{theorem}
The scheme $\Shim^{ss}_{U^{p}}$ is pure of dimension $1$. For $U^{p}$ sufficiently small, the irreducible components are isomorphic to the projective line $\PP^{1}$. The intersection of two irreducible components is a superspecial  point.
\end{theorem}

\begin{proof}
This follows from the main result for the Rapoport-Zink space in the $p$-split case Theorem \ref{main-p-split}.
\end{proof}

\section{Rapoport-Zink Space of PEL type} 
 
\subsection{The Rapoport-Zink space in the $p$ inert case} In the case $p$ is inert, then $E_{p}=\QQ_{p^{2}}$ and we denote by $\psi_{0}, \psi_{1}$ the two embeddings of $\QQ_{p^{2}}$ in $K_{0}$. Recall the functor $\calN$ defined in Subsection \ref{PEL-case}. For $S\in \Nilp$, a $S$-valued point of $\mathcal{N}$ is given by a quadruple $(X, \lambda_{X}, \iota_{X}, \rho_{X})$ where
 \begin{itemize}
\item $X$ is a two dimensional $p$-divisible group over $S$;
\item $\lambda_{X}: X\rightarrow X^{\vee}$ is a principal polarization over $S$;
\item $\iota_{X}: \ZZ_{p^{2}}\rightarrow \End_{S}(X)$ satisfies $$\chara(\iota(a); \Lie(X))= (T-\psi_{0}(a))^{2}(T-\psi_{1}(a))^{2}$$ in $\calO_{S}[T]$;
\item $\rho_{X}: X\times \bar{S}\rightarrow \XX\times \bar{S} $ is a quasi-isogeny where $\bar{S}$ is the locus of $S$ where $p=0$.
\end{itemize}

Let $N$ be the  height $8$ isocrystal associated to $\XX$. Let $V$ be the split Hermitian space of dimension 4 over $\QQ_{p^{2}}$ and consider the quasi-split unitary group $G=\GU(V)$.  We denote by $b\in B(G)$ the $\sigma$-conjugacy class associated to $N$.  We would like to have a concrete description of the group of automorphisms of the isocrystal $N$ that preserve the additional structures that is the group 
\begin{equation}
J_{b}(\QQ_{p})=\{g\in G(K_{0}): g^{-1}b\sigma(g)=b\}.  
\end{equation}
We set $\tau= V^{-1}F$ where $F$ is the Frobenius operator on $N$ and $V^{-1}$ is the operator $p^{-1}F$. Let $N=N_{0}\oplus N_{1}$ where $N_{0}$ is the subspace of $N$ that $\iota(\alpha)$ acts by $\psi_{0}(\alpha)$ and $N_{1}$ is the subspace of $N$ that $\iota(\alpha)$ acts by $\psi_{1}(\alpha)$  for the two embeddings $\psi_{0}, \psi_{1}$. The polarization $\lambda$ induces an alternating form $(\cdot,\cdot)$ on $N$ and $N_{i}$ is totally isotropic with respect $(\cdot, \cdot)$. Then consider the slope zero isocrystal $(N_{0}, \tau)$ and the $\QQ_{p^{2}}$-vector space $C_{0}= N^{\tau=1}_{0}$ of dimension 4.  We define an alternating form $\{\cdot, \cdot\}$ on $N_{0}$ by $\{x, y\}=(x, Fy)$. Similarly for a Dieudonn\'{e} module $M\subset N_{0}$, we can decompose $M=M_{0}\oplus M_{1}$.

\begin{lemma}
The group $J_{b}(\QQ_{p})$ can be identified with $\GU(C_{0})$ the unitary group for the Hermitian space $(C_{0}, \{\cdot,\cdot\})$.
\end{lemma}
\begin{proof}
Since $J_{b}(\QQ)$ respects the $\QQ_{p^{2}}$-action on $N$, it respects the grading $N=N_{0}\oplus N_{1}$.  By definition $J_{b}(\QQ_{p})= \{g\in \End(N)^{\times}_{\QQ}: Fg=gF , (gx, gy)=c(g)(x, y)\}$.  Since $g\in J_{b}(\QQ_{p})$ commutes with the Frobenius $F$, $g$ commutes with $\tau$ and hence it restricts to an action on $C_{0}$ and respect the Hermitian form $\{\cdot,\cdot\}$ up to scalar. 
\end{proof}

\subsection{Set structure in the $p$ inert case} First we would like to describe the set $\calN(\FF)$ by passing from $p$-divisible groups to Dieudonn\'{e} modules. Denote by $\calN(i)$ the open and closed sub formal scheme of $\calN$ defined by the condition that the quasi-isogeny $\rho_{X}$ has height $i$. Then by \cite[Proposition 1.18]{Vol-can10}  we have an isomorphism between $\calN(i)$ and $\calN(0)$. Thus it suffices to describe $\calN(0)$.  We denote by $\calM$ the underlying reduced scheme $\calN_{red}(0)$ of $\calN(0)$.

\begin{lemma}\label{set-description-pre}
The set $\calM(\FF)$ can be identified with
$$\{M\subset N: pM_{1}\subset^{2} VM_{0} \subset^{2} M_{1},  pM_{0}\subset^{2} VM_{1} \subset^{2} M_{0}, M_{0}=M^{\perp}_{1}, M_{1}=M^{\perp}_{0}\}$$
here $\perp$ is the integral dual with respect to $(\cdot, \cdot)$.
\end{lemma}
\begin{proof}
The conditions are translations, in terms of Dieudonn\'{e} modules, of the conditions on $p$-divisible groups defining the functor $\calN$. For example,  $pM_{1}\subset^{2} VM_{0} \subset^{2} M_{1}$ and  $pM_{0}\subset^{2} VM_{1} \subset^{2} M_{0}$ are direct translations of the Kottwitz conditions.
\end{proof}

Notice that $M_{0}$ determines $M_{1}$ and therefore we arrive at the following simpler description of $\calM$.
\begin{lemma}\label{set-description}
The set $\calM(\FF)$ can be identified with the following set of lattices in $N_{0}$
$$\{D \subset N_{0}: pD^{\vee}\subset^{2} D\subset^{2} D^{\vee} \}.$$
Here $D^{\vee}$ is the intergal dual of $D$ with respect to the form $\{\cdot,\cdot\}$.
\end{lemma}
\begin{proof}
Given $M\in \calM(\FF)$, we let $D=M_{0}$. First $pM_{1}\subset^{2} VM_{0} \subset^{2} M_{1}$ is equivalent to $FM_{1}\subset^{2} M_{0} \subset^{2} \frac{1}{p}FM_{1}$. Notice  
\begin{align*}\label{tau-dual}
 F M_{1}=  F M_{0}^{\perp} &=\{ x\in N_{1}: ( M_{0},F^{-1} x)\subset W_{0} \}\\
     &=\{x\in N_{1}: (M_{0}, \frac{V}{p} x)\subset W_{0} \}\\
     &=p M^{\vee}_{0}
 \end{align*}
 and  hence $pM^{\vee}_{0}\subset^{2} M_{0} \subset^{2} M_{0}^{\vee}$. Conversely if we set $M_{0}=D$ and $M_{1}=F^{-1}(pD^{\vee})$, then one can check $M=M_{0}\oplus M_{1}$ does give a point in $\calM(\FF)$ by checking it satisfies all the conditions in Lemma \ref{set-description-pre}.
 \end{proof}

\subsection{Vertex lattices in the $p$-inert case} We will decompose the Rapoport-Zink space $\calM$ into a union of $\calM_{L}$ which we call a \emph{lattice stratum}. These strata are indexed by the so called \emph{vertex lattices} which we will now define. First we need the following lemma. 

\begin{lemma}[ {\cite[Proposition 1.15]{Vol-can10}} ]\label{Hermitian_form}
There exists a basis of $C_{0}$ such that the form $\{\cdot, \cdot\}$ with respect to this basis is given by a unit multiple of the identity matrix.
\end{lemma}

\begin{definition}
A $\ZZ_{p^{2}}$-lattice $L\subset C_{0} $ is called a vertex lattice if $pL^{\vee}\subset L\subset L^{\vee}$.
\end{definition}
 
An elementary computation using a fixed basis of $C_{0}$ and Lemma \ref{Hermitian_form} shows that the index of $L$ defined as the $\ZZ_{p^{2}}$-length of $L/pL^{\vee}$ is an even integer between $0$ and $4$. Therefore there are $3$ types of vertext lattices:

\begin{enumerate}
\item A vertex lattice $L_{1}$: $pL_{1}^{\vee}\subset^{4} L_{1}\subset^{0} L_{1}^{\vee}$ of type $1$;
\item A vertex lattice $L_{02}$: $pL_{02}^{\vee}\subset^{2} L_{02}\subset^{2} L_{02}^{\vee}$ of type $\{0,2\}$;
\item A vertex lattice $L_{3}$: $pL_{3}^{\vee}\subset^{0} L_{3}\subset^{4} L_{3}^{\vee}$ of type $3$.
\end{enumerate}

\begin{remark}\label{diagram-inert}
We would like to explain our choice of terminology for the vertex lattices. The affine Dynkin diagram of type $\tilde{A}_{3}$ is a cycle with $4$-vertices. 
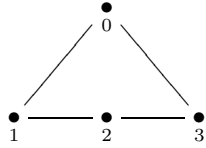
\begin{displaymath}
  \xymatrix{& &\underset{0}\bullet \ar@{-}[dl]\ar@{-}[dr]& \\
                 &\underset{1}\bullet \ar@{-}[r] &\underset{2}\bullet &\underset{3}\bullet\ar@{-}[l]}
\end{displaymath}
The derived group $\tilde{J}$  of $J_{b}$ is $\SU(C_{0})$ and the automorphism of the Dynkin diagram corresponds this group fixes the nodes $1$ and $3$ while exchanges the nodes $0$ and $2$. In particular, in the rational Bruhat-Tits building of $\tilde{J}$, $\{0,2\}$ should be viewed as a single vertex of a fixed standard alcove.
\end{remark}

\begin{definition}\label{lattice-stratum-set} To each vertex lattice $L$, we define the following lattice stratum $\calM_{L}(\FF)$ 
\begin{enumerate}
\item For each vertex lattice $L_{1}$ of type $1$, we define $\calM_{L_{1}}(\FF)=\{D\in \calM(\FF): D\subset L_{1,W_{0}}\}$. 
\item For each vertex lattice $L_{3}$ of type $3$, we define $\calM_{L_{3}}(\FF)=\{D\in \calM(\FF): L_{3,W_{0}}\subset D\}$. 
\item For each vertex lattice $L_{02}$ of type $\{0,2\}$, we define $\calM_{L_{02}}(\FF)=\{D\in \calM(\FF):  L_{02,W_{0}}\subset D\}$. 
\end{enumerate}
\end{definition}

The following two results show that we can decompose $\calM(\FF)$ into a union of lattice strata defined in Defintion \ref{lattice-stratum-set}. 
\begin{lemma}\label{Pappas}
Let $D\in\calM(\FF)$ as in Lemma \ref{set-description}, then $\dim D+\tau(D)/D\leq 1$. 
\end{lemma}
\begin{proof}
Let $M$ be the {\Dieu} described in Lemma \ref{set-description-pre}. Since $M$ is supersingular, $F: M_{0}/VM_{1}\rightarrow M_{1}/VM_{0} $ is not invertible and thus $\dim FM_{0}+VM_{0}/VM_{0}\leq 1$. The result follows since $D$ corresponds to $M_{0}$.
 \end{proof}

\begin{proposition}\label{chain1}
Given $D\in \mathcal{M}(\FF)$, we can find a $\tau$-stable lattice $L(D)$ that it either fits in
\begin{equation}\label{0-4type}pL(D)^{\vee}\subset pD^{\vee}\subset D\subset L(D)\subset L(D)^{\vee}\subset D^{\vee}\end{equation} 
or it fits in 
\begin{equation}\label{4-0type}pD^{\vee}\subset pL(D)^{\vee}\subset L(D)\subset D\subset D^{\vee}\subset L(D)^{\vee}.\end{equation}
\end{proposition}
\begin{proof}

Suppose $D$ is $\tau$-stable, then there is nothing to prove.  Otherwise, Lemma \ref{Pappas} shows that $D\subset^{1} D+\tau(D)$. Suppose that $D+\tau(D)$ is $\tau$-stable, then we set $L(D)=D+\tau(D)$. Notice that Since $D\in \mathcal{M}(\FF)$, we have \begin{equation}\label{condition}pD^{\vee}\subset D\subset D^{\vee} \text{ and } pD^{\vee}\subset \tau(D)\subset D^{\vee}.\end{equation} Therefore $L(D)=D+\tau(D)\subset L(D)^{\vee}=D^{\vee}\cap \tau(D)^{\vee}$ and  $L(D)$ fits in \eqref{0-4type}.

Now suppose $D+\tau(D)$ is not $\tau$-stable and we claim that $D\cap \tau(D)$ is $\tau$-stable. Indeed, by \eqref{condition}, we have $p\tau(D)^{\vee}\subset^{1} D\cap \tau(D)$ and $p\tau(D)^{\vee}\subset^{1} \tau(D)\cap\tau^{2}(D)$. Hence \begin{equation}\label{three-intersection}D\cap \tau(D)\cap \tau^{2}(D)=p\tau(D)^{\vee}.\end{equation} Now, we consider the family of lattices $L_{j}(D)=D+\tau(D)+\cdots +\tau^{j}(D)$. By \cite{RZ-Aoms} Proposition 2.17, there is a minimal $d$ such that $L_{d}(D)=\tau(L_{d}(D))$ and we set $L(D)=L_{d}(D)$.  By \ref{Pappas} , we deduce that $$\tau(L_{j-2}(D))\subset^{1} L_{j-1}(D)\subset^{1} L_{j}(D)$$ and $$\tau(L_{j-2}(D))\subset^{1} \tau(L_{j-1}(D))\subset^{1} L_{j}(D).$$ It follows that \begin{equation}\label{inter}\tau(L_{j-2}(D))=L_{j-1}(D)\cap \tau(L_{j-1}).\end{equation} By \eqref{condition} $D+\tau(D)\subset D^{\vee}\subset p^{-1}\tau(D)$ and  $\tau(D)+\tau^{2}(D)\subset \tau(D)^{\vee}\subset p^{-1}\tau(D)$. Then we have $D+\tau(D)+\tau^{2}(D)\subset p^{-1}\tau(D)$. Thus $$L_{d}(D)=D+\tau(D)+\cdots +\tau^{d}(D)\subset p^{-1}\tau(D)+p^{-1}\tau^{2}(M)+\cdots p^{-1}\tau^{d-1}(M)\subset p^{-1}L_{d-1}(D).$$ Applying $\tau^{l}$ for any integer $l$ on both sides, we get $L_{d}(D)\subset p^{-1} \bigcap_{l\in \ZZ}\tau^{l}(L_{d-1}(D))$. It follows from \eqref{inter}, $\bigcap_{l\in\ZZ}L_{d-1}(D)= \bigcap_{l\in \ZZ}D$. Hence $$L_{d}(D)\subset p^{-1} \bigcap_{l\in \ZZ}\tau^{l}(L_{d-1}(D))= p^{-1}\bigcap_{l\in \ZZ}\tau^{l}(D).$$ Now we apply \eqref{three-intersection} and we find $$L_{d}(M)\subset p^{-1} \bigcap_{l\in \ZZ}\tau^{l}(L_{d-1}(D))= p^{-1}\bigcap_{l\in \ZZ}\tau^{l}(D)=\bigcap_{l\in \ZZ} \tau^{l}(D)^{\vee}=L_{d}(D)^{\vee}.$$ Then $L_{d}(D)$ satisfies \eqref{0-4type} and it follows that the indices among the lattices has to be $pL_{d}(D)^{\vee}\subset^{1} pD^{\vee}\subset^{2} D\subset^{1} L_{d}(D)\subset^{0} L_{d}(D)^{\vee}\subset^{1} D^{\vee}$. This implies $L_{d}(D)=D+\tau(D)$ and this is a contradiction to the assumption that $D+\tau(D)$ is not $\tau$ stable. This proves in this case $D\cap \tau(D)$ is $\tau$-stable and we set $L(D)=D\cap\tau(D)$. Then one checks easily as before that $L(D)$ satisfies \eqref{4-0type} since $L(D)^{\vee}=D^{\vee}+\tau(D)^{\vee}$.
\end{proof}

It follows from the proof of the previous proposition $L(D)$ gives a vertex lattice $L$ by $L=L(D)^{\tau=1}$.

\begin{definition}
For $i=1, 3, 02$, we  define $$\calM_{\{i\}}(\FF)=\bigcup_{L_{i}}\calM_{L_{i}}(\FF)$$ where $L_{i}$ runs through all the vertex lattices of type $1, 3$ and $\{0,2\}$. We will refer to these as the closed Bruhat-Tits strata of type $i$ of $\calM(\FF)$.
\end{definition}

\begin{corollary}We have a decomposition
$\calM(\FF)=\calM_{\{1\}}(\FF)\cup\calM_{\{3\}}(\FF)\cup \calM_{\{02\}}(\FF)$.
\end{corollary}
\begin{proof}
This follows from the description of $\calM(\FF)$ in Lemma \ref{set-description} and Proposition \ref{chain1}. 
\end{proof}

\subsection{Deligne-Lusztig varieties for unitary groups} Let $G_{0}$ be a connected reductive group over $\FF_{p}$ and let $G=G_{0,\FF}$ be the base change to $\FF$. Let $T$ be a maximal torus contained in $G$ and $B$ a Borel subgroup containing $T$. We can assume $T$ is defined over $\FF_{p}$. Let $W$ be the Weyl group corresponding to $(T, B)$. Then $W$ carries an action by $\sigma$ induced by the Frobenius action on $G$.

Let $\Delta^{*}=\{\alpha_{1}, \cdots, \alpha_{n}\}$ be the simple roots corresponding to $(T, B)$ and let $s_{i}$ be the simple reflection corresponding to the 
root $\alpha_{i}$. For $I\subset \Delta^{*}$, let $W_{I}$ be the subgroup of $W$ generated by $\{s_{i}, i\in I\}$.  Consider $P_{I}$ the corresponding  parabolic subgroup of $G$.  For another set $J\subset \Delta^{*}$, we have a decomposition $$G=\bigcup_{w\in W_{I} \backslash W/ W_{J}}P_{I}wP_{J}.$$ and hence a bijection
$$P_{I}\backslash G/ P_{J} \cong W_{I} \backslash W/ W_{J}.$$
We define the relative position map $$\text{inv}: G/P_{I}\times G/P_{J}\rightarrow W_{I} \backslash W/ W_{J} $$ by sending $(g_{1}, g_{2})$ to $g^{-1}_{1}g_{2}\in P_{I}\backslash G/ P_{J} $. 

\begin{definition}
For a given $w\in W_{I}\backslash W/ W_{\sigma(I)}$, the Deligne-Lusztig variety $X_{P_{I}}(w)$ is the locally closed reduced subscheme of $G/P_{I}$ whose $\FF$-points is described by $$X_{P_{I}}(w)=\{gP_{I}\in G/P_{I}; \mathrm{inv}(g, \sigma(g))=w\}.$$
\end{definition}

Let $G_{0}$ be the unitary group over $\FF_{p}$ defined by a split Hermitian space $(V, [\cdot,\cdot])$ of dimension 4 over $\FF_{p^{2}}$.  The group $G=G_{0}\otimes \FF$ can be identified with $\GL(4)$ over $\FF$. We choose $(T, B)$ such that the torus $T$ corresponding to the diagonal matrices and $B$ the Borel corresponding  to the uppertriangular matrices. Write $V\otimes_{\FF_{p^{2}}} \FF$ as $V_{\FF}=V_{0}\oplus V_{1}$. Then $G=\GL(V_{0})$. Define the parabolics $P_{1,3}$ with Levi subgroup $\GL(1)\times \GL(3)$ and $P_{3, 1}$ with Levi subgroup $\GL(3)\times \GL(1)$. Equip $V_{0}$ with the form $(x\otimes a, y\otimes b)=[x, y]\otimes ab^{p}$. For $U\subset V_{0}$ a subspace, we denote by $U^{\perp}$ the orthogonal complement of $U$ with respect to $(\cdot,\cdot)$. We can choose a basis of $V_{0}$ denoted by $\{e_{1}, e_{2}, e_{3}, e_{4}\}$. We consider the elements in $s_{1}=(1,2)$ and $s_{2}=(2, 3)$ and denote by $w_{1}=s_{1}$ and $w_{2}=s_{1}s_{2}$. We are concerned with the following subvariety of the flag variety $G/P_{1,3}$
\begin{equation}\label{DL_1}
Y^{(+)}_{V_{0}}=\{U\subset V_{0}: \dim_{\FF}U=1, U\subset U^{\perp}\};
\end{equation}
which is also equivalent to the following subvariety of $G/P_{3,1}$ by taking the orthogonal complement of $U$
\begin{equation}\label{DL_3}
Y^{(-)}_{V_{0}}=\{U\subset V_{0}: \dim_{\FF}U=3, U^{\perp}\subset U\}.
\end{equation}

Note that by choosing a suitable coordinate, $Y^{(+)}_{V_{0}}\cong Y^{(-)}_{V_{0}}$ are isomorphic to the Fermat hypersurface $x_{0}^{p+1}+x_{1}^{p+1}+x^{p+1}_{2}+x^{p+1}_{3}=0$. 

We can decompose  $Y^{(\pm)}_{V_{0}}$ into locally closed stratums as in \cite[Theorem 2.15]{Vol-can10}.

\begin{theorem}\label{DL-stratification-inert}
There is a decomposition of $Y^{(-)}_{V_{0}}$ into locally closed subvarieties $$Y^{(-)}_{V_{0}}= X_{P_{3,1}}(1)\sqcup X_{B}(w_{1})\sqcup X_{B}(w_{2}).$$ The subvariety $ X_{P_{3, 1}}(1)$ is closed of dimension $0$, the subvariety $X_{B}(w_{1})$ is of dimension $1$ and an irreducible component of its closure is a $\PP^{1}$, the subvariety $X_{B}(w_{2})$ is open and dense in $Y^{(-)}_{V_{0}}$. Similarly we have $$Y^{(+)}_{V_{0}}= X_{P_{1,3}}(1)\sqcup X_{B}(w_{1})\sqcup X_{B}(w_{2}).$$ 
\end{theorem}
\begin{proof}

By Definition \ref{DL_3},  $X_{P_{3,1}}(1)$ is the subvariety consisting of those $U$ that is $\sigma$-invariant and $X_{P_{3,1}}(w_{1})$ consists of $U$ that is not $\sigma$-invariant. Therefore we have $Y^{(-)}_{V_{0}}= X_{P_{3,1}}(1)\sqcup X_{P_{3,1}}(w_{1})$. We can further decompose $X_{P_{3,1}}(w_{1})$ by $X_{P_{3,1}}(w_{1})=X_{B}(w_{1})\sqcup X_{B}(w_{2})$. Indeed the two dimensional stratum $X_{B}(w_{2})$ consist of all $U$ of dimension $3$ in $V_{0}$ such that for $\tau=\sigma^{2}$, $U\cap \tau(U)$ is of dimension $2$ and $U\cap \tau(U)\cap \tau^{2}(U)$ is of dimension $1$. The one dimensional stratum $X_{B}(w_{1})$ consists of all $U$ of dimension $3$ such that $U\cap \tau(U)$ is a $\tau$-invariant plane. 

Notice that for $U\in X_{B}(w_{1})$, we have $U\cap\tau(U)= (U\cap\tau(U))^{\perp}$ and we have for a chain 
$$U^{\perp}\subset U\cap \tau(U)\subset U$$ with $U\cap\tau(U)$ a $\tau$-invariant plane and $U^{\perp}$ gives a point on $\PP^{1}(U\cap \tau(U))$.
Conversely for each maximal rational isotropic plane $U^{\prime}\subset V_{0}$ and a three dimensional $U$ that $U^{\prime}\subset U$, we have $U\in X_{B}(w_{1})$. In fact, taking the dual of the inclusion $U^{\prime}\subset U$ gives $U^{\perp}\subset U^{\prime\perp}=U^{\prime}\subset U$.   Therefore each irreducible component of the one dimensional strata is a $\PP^{1}$.  
\end{proof}

\subsection{Ekedahl-Oort stratifications of lattice strata}

For a vertex lattice $L_{1}$ of type $1$ and a vertex lattice $L_{3}$ of type $3$, we define the set theoretic \emph{Ekedahl-Oort stratification} of $\calM_{L_{1}}(\FF)$ and $\calM_{L_{3}}(\FF)$ in this subsection. We set

\begin{equation}\label{EO-L1}
\begin{split}
&\mathcal{M}^{\circ}_{L_{1}}(\FF)=\mathcal{M}_{L_{1}}(\FF)-( \mathcal{M}_{\{02\}}(\FF)\cup \mathcal{M}_{\{3\}}(\FF));\\
&\calM_{L_{1},\{3\}}(\FF)=\calM_{L_{1}}(\FF)\cap \calM_{\{3\}}(\FF);\\
&\calM^{\circ}_{L_{1},\{3\}}(\FF)=(\calM_{L_{1}}(\FF)\cap \calM_{\{3\}}(\FF))- \calM_{\{0, 2\}}(\FF);\\
&\calM_{L_{1},\{0, 2\}}(\FF)=\calM_{L_{1}}(\FF)\cap \calM_{\{0, 2\}}(\FF). \\
\end{split}
\end{equation}
We call the decomposition $$\calM_{L_{1}}(\FF)= \mathcal{M}^{\circ}_{L_{1}}(\FF)\sqcup \calM^{\circ}_{L_{1},\{3\}}(\FF)\sqcup \calM_{L_{1},\{0, 2\}}(\FF)$$ the \emph{Ekedahl-Oort stratification} for $\calM_{L_{1}}(\FF)$. Let $V(L_{1})$ be the Hermitian space defined by $L_{1,W_{0}}/pL^{\vee}_{1,W_{0}}$ with the Hermitian form induced by $\{\cdot,\cdot\}$. Notice that the square of the Frobenius $\tau=\sigma^{2}$ on $V(L_{1})$ is induced by the operator $\tau$ on $L_{1, W_{0}}$ and we abuse the notation here. The following theorem describes the above stratification in terms of Deligne-Lusztig varieties. 

\begin{proposition}\label{DL1}
There is a bijection between $\mathcal{M}_{L_{1}}(\FF)$ and $Y^{(-)}_{V(L_{1})}(\FF)$. This bijection is compatible with the stratification on both sides in the sense that
\begin{enumerate}
\item There is a bijection between $\calM^{\circ}_{L_{1}}(\FF)$ and $X_{B}(w_{2})(\FF)$;
\item There is a bijection between $\calM^{\circ}_{L_{1},\{3\}}(\FF)$ and $X_{B}(w_{1})(\FF)$;
\item There is a bijection between $\calM_{L_{1}, \{0, 2\}}(\FF)$ and $X_{P_{3,1}}(1)(\FF)$.
\end{enumerate}
\end{proposition}

\begin{proof}
Given a point $D\in \calM_{L_{1}}(\FF)$  with $D$ described as in Lemma \ref{set-description}. Sending $D$ to $U=pD^{\vee}/pL^{\vee}_{1, W_{0}}$ gives a map from $f_{1}: \calM_{L_{1}}(\FF) \rightarrow Y^{(-)}_{V(L_{1})}(\FF)$. Conversely given $U\subset L_{1, W_{0}}/pL^{\vee}_{1, W_{0}}$, let $D$ be the preimage of $U$ under the natural reduction map. Then $pD^{\vee}$ is the preimage of $U^{\perp}$. Since $U^{\perp}\subset U$, we have $pL^{\vee}_{1, W_{0}}\subset pD^{\vee}\subset D\subset L_{1, W_{0}}$.  Then sending $U\in Y^{(-)}_{V(L_{1})}(\FF)$ to $U^{\perp}\subset U$ gives a map  $q: Y^{(-)}_{V(L_{1})}(\FF)\rightarrow \calM_{L_{1}}(\FF)$ which is clearly the inverse map to the previous one. This shows the bijection between $\calM_{L_{1}}(\FF)$ and $Y^{(-)}_{V(L_{1})}(\FF)$. 

Given $D\in \calM^{\circ}_{L_{1}}(\FF)$, then by definition $D\cap \tau(D)$ is not $\tau$-invariant. Therefore $U\cap \tau(U)$ is not $\tau$-invariant and hence $U\cap\tau(U)\cap\tau^{2}(U)$ is one dimensional. Therefore the image of $D$ under the map constructed previously lands in $X_{B}(w_{2})(\FF)$. Conversely given $U\in X_{B}(w_{1})$ then the its image under $q$ is clearly in $\calM^{\circ}_{L_{1}}(\FF)$.

Given $D\in \calM^{\circ}_{L_{1},\{3\}}(\FF)$, then by definition $D\cap \tau(D)$ is $\tau$-invariant. Therefore $U\cap \tau(U)$ is $\tau$-invariant. Therefore the image of $D$ under the map constructed previously lands in $X_{B}(w_{1})(\FF)$. Conversely given $U\in X_{B}(w_{1})$ then the its image under $q$ is clearly in $\calM^{\circ}_{L_{1},\{3\}}(\FF)$.

Similarly if $D\in \calM_{L_{1},\{0, 2\}}(\FF)$, then $D=L_{02,W_{0}}$ for some vertex lattice $L_{02}$ of type $\{0, 2\}$. Therefore $D$ is $\tau$-invariant and hence $U=D/pL^{\vee}_{1}$ is also $\tau$-invariant and its image is in $X_{P_{\{3,1\}}}(1)$.  Conversely given $U\in X_{P_{\{3,1\}}}(1)(\FF)$, then its image under $q$ is in 
$ \calM_{L_{1},\{0, 2\}}(\FF)$. This finishes the proof of all the assertions.
\end{proof}

The decomposition of $\calM_{L_{3}}$ for a vertex lattice of type $3$ is completely parallel to the previous case. We set 
\begin{equation}\label{EO-L3}
\begin{split}
&\mathcal{M}^{\circ}_{L_{3}}(\FF)=\mathcal{M}_{L_{3}}(\FF) -( \mathcal{M}_{\{1\}}(\FF)\cup \mathcal{M}_{\{0, 2\}}(\FF));\\
&\calM_{L_{3},\{1\}}(\FF)=\calM_{L_{3}}(\FF)\cap \calM_{\{1\}}(\FF);\\
&\calM^{\circ}_{L_{3},\{1\}}(\FF)=(\calM_{L_{3}}(\FF)\cap \calM_{\{1\}}(\FF))- \calM_{\{0, 2\}}(\FF);\\
&\calM_{L_{3},\{0, 2\}}(\FF)=\calM_{L_{3}}(\FF)\cap \calM_{\{0, 2\}}(\FF).\\ 
\end{split}
\end{equation}
We call the decomposition $$\calM_{L_{3}}(\FF)= \mathcal{M}^{\circ}_{L_{3}}(\FF)\sqcup \calM^{\circ}_{L_{3},\{1\}}(\FF)\sqcup \calM_{L_{3},\{0, 2\}}(\FF)$$ the set theoretic \emph{Ekedahl-Oort stratification} for $\calM_{L_{3}}(\FF)$. Let $V(L_{3})$ be the Hermitian space defined by $L^{\vee}_{3,W_{0}}/L_{3, W_{0}}$ with the Hermitian form induced by $\{\cdot, \cdot\}$. 

\begin{proposition}\label{DL3}
There is a bijection between $\mathcal{M}_{L_{3}}(\FF)$ and $Y^{(+)}_{V(L_{3})}(\FF)$. This bijection is compatible with the stratification on both sides in the sense that
\begin{enumerate}
\item  There is a bijection between $\calM^{\circ}_{L_{3}}(\FF)$ and $X_{B}(w_{2})(\FF)$;
\item  There is a bijection between $\calM^{\circ}_{L_{3},\{1\}}(\FF)$ and $X_{B}(w_{1})(\FF)$;
\item  There is a bijection between $\calM_{L_{3}, \{0, 2\}}(\FF)$ and $X_{P_{\{1,3\}}}(1)(\FF)$.
\end{enumerate}
\end{proposition}

\begin{remark}
Using the theory of windows of display in \cite{Zink-pro99} in place of {\Dieu}, the bijections proved in Proposition \ref{DL1} and Proposition \ref{DL3} can be extended to any field extension $k$ over $\FF$.
\end{remark}

For a vertex lattice $L_{02}$ of type $\{0, 2\}$, recall that the stratum $\mathcal{M}_{L_{02}}(\FF)$ is the set $$\{D\in \mathcal{M}(\FF): L_{02,W_{0}}\subset D\}.$$
\begin{proposition}\label{L2}
The stratum $\mathcal{M}_{L_{02}}(\FF)$ consists of a superspecial point  and $\calM_{\{0, 2\}}(\FF)$ consists of all the superspecial points.
\end{proposition}
\begin{proof}
Given $D\in \mathcal{M}_{L_{02}}(\FF)$, we have the following chain
$$L_{02,W_{0}} \subset D\subset D^{\vee}\subset L^{\vee}_{02,W_{0}}.$$
This forces $L_{02,W_{0}} = D$ by considering the index. Since $\tau(L_{02,W_{0}})=L_{02, W_{0}}$, it follows that $VD=F D$. Hence $M=D\oplus VD^{\vee}$ is a superspecial $\GU(2,2)$ type Dieudonn\'{e} module as described in Lemma \ref{set-description-pre}. It is also clear that all superspecial points in $\calM(\FF)$ arises this way.
\end{proof}

\begin{lemma}
Let $L_{02}$ be a vertex lattice of type $\{0,2\}$. Then we can find a vertex lattice $L_{1}$ of type $1$ containing $L_{02}$ and a vertex lattice $L_{3}$ of type $3$ contained in $L_{02}$. 
\end{lemma}
\begin{proof}
Since $L_{02}$ is a vertex lattice of type $\{0, 2\}$, $pL_{02}^{\vee}\subset^{2}L_{02}\subset^{2} L_{02}^{\vee}$. Consider the space $L_{02}^{\vee}/L_{02}$ of dimension two. Choose any line $\bar{l}$ in $L^{\vee}_{02}/L_{02}$ and denote by $L_{1}$ its preimage in $L^{\vee}_{02}$ under the natural reduction map. One verifies immediately that $L_{1}$ is a vertex lattice of type $1$. Similarly choose any line in $L_{02}/pL^{\vee}_{02}$ gives a vertex $L_{3}$ lattice of type $3$. 
\end{proof}

This lemma shows that any superspecial point is contained in the intersection $\calM_{L_{1}}(\FF)\cap \calM_{L_{3}}(\FF)$ for a vertex  lattice $L_{1}$ of type $1$ and a vertex lattice $L_{3}$ of type $3$.  

Finally we study the intersection of any two $2$-dimensional strata. We have already seen that for a $1$-vertex $L_{1}$ and a $3$-vertex $L_{3}$, if the intersection $\calM_{L_{1}}(\FF)\cap \calM_{L_{3}}(\FF)$ is non-empty then it equal to  $\PP^{1}(\FF)$. Let $L_{1}$ and $L^{\prime}_{1}$ be two distinct vertex lattices of type $1$.  Similarly let $L_{3}$ and $L^{\prime}_{3}$ be two distinct vertex lattices of type $3$. 

\begin{lemma}
Suppose the intersection of $\calM_{L_{1}}(\FF)$ and $\calM_{L^{\prime}_{1}}(\FF)$ is nonempty, then it consists of a superspecial point corresponding to the vertex lattice $L_{1}\cap L^{\prime}_{1}$ of type $\{0, 2\}$. Similarly if the intersection of $\calM_{L_{3}}(\FF)$ and $\calM_{L^{\prime}_{3}}(\FF)$ is nonempty, then it consists of a superspecial point corresponding to the vertex lattice $L_{3}+ L^{\prime}_{3}$ of type $\{0, 2\}$.
\end{lemma}

\begin{proof}
Suppose $M\in \calM_{L_{1}}\cap \calM_{L^{\prime}_{1}}$. Then $pL_{1,W_{0}}\subset^{1} pM^{\vee}\subset^{2} M\subset^{1} L_{1,W_{0}}$ and $pL^{\prime}_{1, W_{0}}\subset^{1} pM^{\vee}\subset^{2} M\subset^{1} L^{\prime}_{1, W_{0}}$, this forces $M=L_{1, W_{0}}\cap L^{\prime}_{1, W_{0}}$. Moreover $L_{1, W_{0}}\cap L^{\prime}_{1, W_{0}}$ is a vertex lattice of type $\{0,2\}$ since it fits in the chain $$L^{\vee}_{1, W_{0}}+L^{\prime \vee}_{1, W_{0}}\subset pM^{\vee}\subset M\subset L_{1, W_{0}}\cap L^{\prime}_{1, W_{0}}.$$
The statement for $L_{3}, L^{\prime}_{3}$ is proved in the exact same way.
\end{proof}

\subsection{The isogeny trick} We have studied the set theoretic structure of the lattice stratum $\calM_{L}(\FF)$ in the previous sections. The goal in this subsection is to equip $\calM_{L}(\FF)$ with a scheme theoretic structure and translates the stratification of $\calM_{L}(\FF)$ as in Proposition \ref{DL1} and Proposition $\ref{DL3}$ to a scheme theoretic level. We follow the method of \cite[Section 3]{VW-invent11} which we call the isogeny trick. 

Let $L_{1}$ be a vertex lattice of type $1$.  We define $L^{+}_{1}=L_{1}\oplus V^{-1}L_{1}$ and $L^{-}_{1}=(L^{+}_{1})^{\perp}$ where $\perp$ is the dual taken with respect to the form $(\cdot,\cdot)$ inside of $N$. By an easy computation, we have $L^{\perp}_{1}=VL^{\vee}_{1}$ and $(V^{-1}L_{1})^{\perp}=pL^{\vee}_{1}$. Therefore $L^{-}_{1}=pL^{\vee}_{1}\oplus VL^{\vee}_{1}$.

\begin{lemma}\label{p-div-L1}
The lattices $L_{1}^{+}$ and $L_{1}^{-}$ gives rise to $p$-divisible groups $\XX_{L_{1}^{+}}$ and $\XX_{L_{1}^{-}}$ with quasi-isogenies $\rho_{L^{+}_{1}}$ and $\rho_{L^{-}_{1}}$ to $\XX$. Moreover there is a commutative diagram:
$$
\begin{tikzcd}
 \XX_{L_{1}^{+}} \arrow[r,  "\sim"] \arrow[d,  "\rho_{L_{1}^{+}}"]
    & \XX^{\vee}_{L_{1}^{-}} \\
   \XX \arrow[r,  "\lambda_{\XX}"] & \XX^{\vee}  \arrow[u,  "\rho^{\vee}_{L_{1}^{-}}"] .
\end{tikzcd}$$ 
\end{lemma}

\begin{proof}
To see $L_{1, W_{0}}^{+}$ defines a $p$-divisible groups, we need to verify that $p L_{1, W_{0}}^{+} \subset V L_{1, W_{0}}^{+} \subset L_{1, W_{0}}^{+}$. This is equivalent to verify that $pV^{-1}L_{1, W_{0}}\subset V L_{1, W_{0}}\subset V^{-1} L_{1, W_{0}}$ and $pL_{1, W_{0}}\subset V V^{-1} L_{1, W_{0}}\subset  L_{1, W_{0}}$. Notice that $V^{2}=p$ on $L_{1, W_{0}}$ and then the two equalities are clear, therefore $L^{+}_{1, W_{0}}$ indeed gives rise to a Dieudonn\'{e} module.  The quasi-isogenies are given by the inclusion $L_{1, W_{0}}\subset N_{0}$ and the commutativity of the diagram follows from construction. \end{proof}

For any $\FF$-algebra $R$, let  $(X, \iota_{X},  \lambda_{X}, \rho_{X})$ be an $R$-point of $\calM_{L_{1}}$. Consider the quasi-isogenies defined by 
\begin{equation}\label{rho0+}
\rho_{X,L_{1}^{+}} :X \xrightarrow{\rho_{X}} \XX_{R} \xrightarrow{\rho^{-1}_{L_{1}^{+}}} \XX_{L_{1}^{+}, R} 
 \end{equation}
and
\begin{equation}\label{rho0-}
\rho_{X,L_{1}^{-}} :\XX_{L_{1}^{-}, R} \xrightarrow{\rho_{L_{1}^{-}}} \XX_{R} \xrightarrow{\rho_{X}^{-1}} X.
 \end{equation}

We define the lattice stratum $\calM_{L_{1}}$ associated to $L_{1}$ as the subfunctor of $\calM$ consisting of those $(X, \iota_{X},  \lambda_{X}, \rho_{X})$ over $R$ such that $\rho_{X,L_{1}^{+}}$ is an isogeny. This is also equivalent to $\rho_{X,L_{1}^{-}}$ being an isogeny.

\begin{proposition}\label{points-of-L1}
The set $\calM_{L_{1}}(\FF)$ as in Definition \ref{lattice-stratum-set} is precisely set of $\FF$-points of $\calM_{L_{1}}$.
\end{proposition}
\begin{proof}
Giving a $\FF$-point of $\calM_{L_{1}}$ is equivalent to giving a Dieudonn\'{e} module $M\subset L^{+}_{1, W_{0}}$ but this is equivalent to $M_{0}\subset L_{1, W_{0}}$. Indeed, if $M\subset L^{+}_{1, W_{0}}$, then $M_{0}\subset L_{1, W_{0}}$ is obvious. Conversely if $M_{0}\subset L_{1, W_{0}}$, then $M_{1}= V M^{\vee}_{0}\subset F^{-1}M_{0}\subset F^{-1}L_{1, W_{0}}=V^{-1}L_{1, W_{0}}$.
\end{proof}

Let $L_{3}$ be a vertex lattice of type $3$.  We define the $L^{+}_{3}=L_{3}\oplus V L_{3}$ and $L^{-}_{3}=(L^{+}_{3})^{\perp}$ where $\perp$ is the dual taken with respect to the form $(\cdot,\cdot)$ inside of $N_{0}$. By an easy computation, we have $L^{\perp}_{3}=V(L^{\vee}_{3})$ and $(V L_{1})^{\perp}=L^{\vee}_{3}$. Therefore $L^{-}_{3}=L^{\vee}_{3}\oplus V (L^{\vee}_{3})$.  Using the same argument in Proposition \ref{p-div-L1}, we can prove the following result.
%Hence $L^{-}_{3}=pL_{3}\oplus VL_{3}$.

\begin{lemma}\label{p-div-L3}
The lattices $L_{3}^{+}$ and $L_{3}^{-}$ gives rise to a $p$-divisible groups $\XX_{L_{3}^{+}}$ and $\XX_{L_{3}^{-}}$ with quasi-isogenies $\rho_{L^{+}_{3}}$ and $\rho_{L^{-}_{3}}$ to $\XX$. Moreover there is a commutative diagram:
$$
\begin{tikzcd}
 \XX_{L_{3}^{+}} \arrow[r,  "\sim"] \arrow[d,  "\rho_{L_{3}^{+}}"]
    & \XX^{\vee}_{L_{3}^{-}} \\
   \XX \arrow[r,  "\lambda_{\XX}"] & \XX^{\vee}  \arrow[u,  "\rho^{\vee}_{L_{3}^{-}}"] .
\end{tikzcd}$$ 
\end{lemma}

For a $\FF$-algebra $R$, let  $(X, \iota_{X},  \lambda_{X}, \rho_{X})$ be a $R$-point of $\calM_{L_{1}}$. Consider the quasi-isogenies defined by 

\begin{equation}\label{rho3+}
\rho_{X,L_{3}^{+}} :\XX_{L_{3}^{+}, R} \xrightarrow{\rho_{L_{3}^{+}}} \XX_{R} \xrightarrow{\rho_{X}^{-1}} X;
 \end{equation} 
and 
\begin{equation}\label{rho3-}
\rho_{X,L_{3}^{-}} :X \xrightarrow{\rho_{X}} \XX_{R} \xrightarrow{\rho^{-1}_{L_{3}^{-}}} \XX_{L_{3}^{-}, R} 
 \end{equation}
 
We define the lattice stratum $\calM_{L_{3}}$ associated to $L_{3}$ as the subfunctor of $\calM$ consisting of those $(X, \iota_{X},  \lambda_{X}, \rho_{X})$ over $R$ such that $\rho_{X,L_{3}^{+}}$ is an isogeny. This is also equivalent to $\rho_{X,L_{3}^{-}}$ is an isogeny. 

%\begin{definition}
%We define the $i$-vertex stratum $\calM_{\{i\}}$ of $\calM$ by the union $\bigcup_{L_{i}}\calM_{\{i\}}$ where $L_{i}$ runs through all the $i$-vertex lattices.
%\end{definition}

\subsection{The isomorphism between $\calM_{L}$ and $Y^{(\pm)}_{V(L)}$} Let $L_{1}$ be a vertex lattice of type $1$ and $L_{3}$ be a vertex lattice of type $3$. The goal of this subsection is to define maps $f_{1}: \calM_{L_{1}}\rightarrow Y^{(-)}_{V(L_{1})}$ and $f_{3}: \calM_{L_{3}}\rightarrow Y^{(+)}_{V(L_{3})}$  and show that they are in fact isomorphisms. 

We begin with $L_{1}$. Given a point $(X, \iota_{X},  \lambda_{X}, \rho_{X})\in \calM_{L_{1}}(R)$ for a $\FF$-algebra $R$. Consider the composite isogeny $$\rho_{L_{1}}:\XX_{L_{1}^{-}, R} \xrightarrow{\rho_{X,L_{1}^{-}}} X  \xrightarrow{\rho_{X,L_{1}^{+}}}\XX_{L_{1}^{+}, R}.$$
The kernel of $\rho_{L_{1}}$ is the quotient $L^{+}_{1, R}/L^{-}_{1, R}$ which can be identified with $L_{1,R}/pL^{\vee}_{1,R}\oplus V^{-1}L_{1,R}/V L^{\vee}_{1,R}$. The kernel of $\rho_{X,L^{-}_{1}}$ is a direct summand of $L_{1, R}/pL^{\vee}_{1, R}\oplus V^{-1}L_{1, R}/V L^{\vee}_{1, R}$ which we compute to be $M_{0}/pL^{\vee}_{1, R}\oplus V( M^{\vee}_{0})/V L^{\vee}_{1, R}$ where $M=M_{0}\oplus M_{1}$ is the covariant Dieudonn\'{e} crystal $\DD(X)(R)$ of $X$ evaluated at $R$. Sending $(X, \iota_{X},  \lambda_{X}, \rho_{X})\in \calM_{L_{1}}(R)$ to $M_{0}/ pL^{\vee}_{1, R}$ gives a map from $\calM_{L_{1}}$ to $Y^{(-)}_{V(L_{1})}$ with $V(L_{1})=L_{1, W_{0}}/pL^{\vee}_{1, W_{0}}$. We denote this map by $f_{1}$.

Now we consider $L_{3}$. Given a point $(X, \iota_{X},  \lambda_{X}, \rho_{X})\in \calM_{L_{3}}(R)$ for an $\FF$-algebra $R$. Consider the isogenies $$\rho_{L_{3}}:\XX_{L_{3}^{+}, R} \xrightarrow{\rho_{X,L_{3}^{+}}} X  \xrightarrow{\rho_{X,L_{3}^{-}}}\XX_{L_{3}^{-}, R}.$$
The kernel of $\rho_{L_{3}}$ is the quotient $L^{-}_{3, R}/L^{+}_{3, R}$ which we compute to be $L^{\vee}_{3,R}/L_{3,R}\oplus V L^{\vee}_{3,R}/V L_{3,R}$. The kernel of $\rho_{X,L^{+}_{3}}$ is a direct summand of $L^{\vee}_{3,R}/L_{3,R}\oplus V L^{\vee}_{3,R}/V L_{3,R}$ which we compute to be $M_{0}/L_{3, R}\oplus V M^{\vee}_{0}/V L_{3, R}$ where $M=M_{0}\oplus M_{1}$ is the covariant Dieudonn\'{e} crystal $\DD(X)(R)$ of $X$. Sending $(X, \iota_{X},  \lambda_{X}, \rho_{X})\in \calM_{L_{3}}(R)$ to $M_{0}/ L_{3, R}$ gives a map from $\calM_{L_{3}}$ to $Y^{(+)}_{V(L_{3})}$ with $V(L_{3})=L^{\vee}_{3, W_{0}}/L_{3, W_{0}}$. We denote this map by $f_{3}$.

\begin{proposition}\label{projectivity-L1}
The functors $\calM_{L_{1}}$ and $\calM_{L_{3}}$ are representable by closed subschemes of $\calM$. Moreover they are  projective schemes.
\end{proposition}
\begin{proof}
We will only give the proof for $\calM_{L_{1}}$. The subfunctor is representable by a closed subscheme of $\calM$  by \cite[Proposition 2.9]{RZ-Aoms}. Moreover $\calM_{L_{1}}$ is bounded in the sense of \cite[2.30]{RZ-Aoms}. This can be seen from the fact that we have construced isogenies
$$\XX_{L_{1}^{-}, R} \xrightarrow{\rho_{X,L_{1}^{-}}} X  \xrightarrow{\rho_{X,L_{1}^{+}}}\XX_{L_{1}^{+}, R}.$$
Hence $\calM_{L_{1}}$ is also quasi-projective \cite[Corollary 2.31]{RZ-Aoms}. Thus $\calM_{L_{1}}$ is closed subscheme of a projective scheme by \cite[Corollary 2.29]{RZ-Aoms}. Thus it is projective it self.
\end{proof}

\begin{theorem}\label{isom0}
The maps $f_{1}: \calM_{L_{1}}\rightarrow Y^{(-)}_{V(L_{1})}$ and  $f_{3}: \calM_{L_{3}}\rightarrow Y^{(+)}_{V(L_{3})}$ are isomorphisms. 
\end{theorem}

\begin{proof}
We have shown that $f_{1}$ is bijective on $k$-points for any field extension $k$ over $\FF$. It then follows that $f_{1}$ is birational, proper, quasi-finite.  Since $Y^{(-)}_{V(L_{1})}$ is a closed subscheme of a projective scheme, it is projective. Moreover it is smooth and hence normal.Then by Zariski's main theorem, $f_{1}$ is an isomorphism.
\end{proof}

\subsection{The Bruhat-Tits stratification in the $p$-inert case} We can now transfer the results on the Bruhat-Tits stratification from the set theoretic level in Proposition \ref{DL1} and Proposition \ref{DL3} to the scheme theoretic level. We define the closed \emph{Bruhat-Tits stratum} of type $i$ for $i=1, 3, \{0,2\}$ by the union $$\calM_{\{i\}}=\bigcup_{L_{i}}\calM_{\{i\}}$$ where $L_{i}$ runs through all the vertex lattices of type $i$.

We define subschemes $\mathcal{M}^{\circ}_{L_{1}}, \calM^{\circ}_{L_{1},\{3\}}, \calM_{L_{1},\{0,2\}}$ of $\mathcal{M}_{L_{1}}$ called the Ekedahl-Oort strata by
\begin{equation}\label{EO-L1}
\begin{split}
&\mathcal{M}^{\circ}_{L_{1}}=\mathcal{M}_{L_{1}}-( \mathcal{M}_{\{0, 2\}}\cup \mathcal{M}_{\{3\}});\\
&\calM_{L_{1},\{3\}}=\calM_{L_{1}}\cap \calM_{\{3\}};\\
&\calM^{\circ}_{L_{1},\{3\}}=(\calM_{L_{1}}\cap \calM_{\{3\}})- \calM_{\{0, 2\}};\\
&\calM_{L_{1},\{0, 2\}}=\calM_{L_{1}}\cap \calM_{\{0,2\}}. \\
\end{split}
\end{equation}
Similarly, we define subschemes $\mathcal{M}^{\circ}_{L_{3}}, \calM^{\circ}_{L_{3},\{1\}}, \calM_{L_{3},\{0, 2\}}$ of $\mathcal{M}_{L_{3}}$ called the Ekedahl-Oort strata by
\begin{equation}\label{EO-L1}
\begin{split}
&\mathcal{M}^{\circ}_{L_{3}}=\mathcal{M}_{L_{3}}-( \mathcal{M}_{\{0, 2\}}\cup \mathcal{M}_{\{1\}});\\
&\calM_{L_{3},\{1\}}=\calM_{L_{3}}\cap \calM_{\{1\}};\\
&\calM^{\circ}_{L_{3},\{1\}}=(\calM_{L_{3}}\cap \calM_{\{1\}})- \calM_{\{0, 2\}};\\
&\calM_{L_{3},\{0, 2\}}=\calM_{L_{3}}\cap \calM_{\{0, 2\}}. \\
\end{split}
\end{equation}

We have the Ekedahl-Oort stratification $$\calM_{L_{1}}= \mathcal{M}^{\circ}_{L_{1}}\sqcup \calM^{\circ}_{L_{1},\{3\}} \sqcup \calM_{L_{1}, \{0,2\}}$$ and  $$\calM_{L_{3}}= \mathcal{M}^{\circ}_{L_{3}}\sqcup \calM^{\circ}_{L_{3},\{1\}} \sqcup \calM_{L_{3},\{0,2\}}.$$

\begin{theorem}\label{stratification}
Let $L_{1}$ be vertex lattice of type $1$. The isomorphism $f_{1}: \calM_{L_{1}}\rightarrow Y^{(-)}_{V(L_{1})}$ is compatible with the stratification on both sides:
\begin{enumerate}
\item $\calM^{\circ}_{L_{1}}$ is isomorphic to $X_{B}(w_{2})$;
\item $\calM^{\circ}_{L_{1},\{3\}}$ is isomorphic to $X_{B}(w_{1})$;
\item  $\calM_{L_{1},\{0, 2\}}$ is isomorphic to $X_{P_{\{3, 1\}}}(1)$.
\end{enumerate}
Let $L_{3}$ be a vertex lattice of type $3$. The isomorphism $f_{3}: \calM_{L_{3}}\rightarrow Y^{(+)}_{V(L_{3})}$ is compatible with the stratification on both sides:
\begin{enumerate}
\item $\calM^{\circ}_{L_{3}}$ is isomorphic to $X_{B}(w_{2})$;
\item $\calM^{\circ}_{L_{3},\{1\}}$ is isomorphic to $X_{B}(w_{1})$;
\item  $\calM_{L_{3},\{0, 2\}}$ is isomorphic to $X_{P_{\{1, 3\}}}(1)$.
\end{enumerate}
\end{theorem}

\begin{proof}
Once we know that $f_{1}$ and $f_{3}$ are isomorphisms, the statements about compatibilities of the stratifications can be checked on $\FF$-points. These have been checked in Proposition \ref{DL1} and Proposition \ref{DL3}.
\end{proof}

Notice that the  Ekedahl-Oort stratification has the desired stratification properties. Indeed, the closure of the open stratum $\mathcal{M}^{\circ}_{L_{1}}$ is the union $\mathcal{M}^{\circ}_{L_{1}}\sqcup \calM^{\circ}_{L_{1},\{3\}} \sqcup \calM_{L_{1},\{0,2\}}$ and the closure of the stratum $\calM^{\circ}_{L_{1},\{3\}}$ is the union $\calM^{\circ}_{L_{1},\{3\}}\sqcup \calM_{L_{1},\{0,2\}}$.

Finally we can put everything together and define the Bruhat-Tits strata of $\calM$ by 
\begin{equation}
\begin{split}
&\calM^{\circ}_{\{1\}}=\bigcup_{L_{1}}\calM^{\circ}_{L_{1}}\text{ where $L_{1}$ runs through all vertex lattices of type $1$};\\
&\calM^{\circ}_{\{3\}}=\bigcup_{L_{3}}\calM^{\circ}_{L_{3}}\text{ where $L_{1}$ runs through all vertex lattices of type $3$};\\ 
&\calM^{\circ}_{\{13\}}=\bigcup_{L_{1}}\calM^{\circ}_{L_{1},\{3\}}\text{ where $L_{1}$ runs through all vertex lattices of type $1$} ;\\
&\calM_{\{0,2\}}=\bigcup_{L_{02}}\calM_{L_{0, 2}}\text{ where $L_{02}$ runs through all vertex lattices of type $\{0, 2\}$}. \\
\end{split}
\end{equation}

\subsection{The main result in the $p$ inert case} We now summarize the results obtained in the previous sections and finish the description of the special fiber for Rapoport-Zink space in the case $p$ is inert.

\begin{theorem}\label{main-result-inert}
The formal scheme $\calN$ can be written as $\calN=\bigcup_{i\in \ZZ} \calN(i)$. The connected components $\calN(i)$ are all isomorphic to $\calN(0)$. Let $\calM=\calN_{red}(0)$, then $\calM$ is pure of dimension $2$.
\begin{enumerate}
\item $\calM$ admits the Bruhat-Tits stratification $$\calM=\calM^{\circ}_{\{1\}}\sqcup \calM^{\circ}_{\{3\}} \sqcup \calM^{\circ}_{\{1,3\}}\sqcup \calM_{\{0,2\}}.$$

\item The closure of $\calM^{\circ}_{\{1\}}$ is $\calM_{\{1\}}=\bigcup_{L_{1}}\calM_{L_{1}}$ where  $L_{1}$ runs through all the vertex lattice of type $1$ and each $\calM_{L_{1}}$ is isomorphic to the Fermat hypersurface $x_{0}^{p+1}+x_{1}^{p+1}+x_{2}^{p+1}+x_{3}^{p+1}=0$. Moreover $\calM_{L_{1}}$ admits a stratification $$\calM_{L_{1}}=\calM^{\circ}_{L_{1}}\sqcup \calM^{\circ}_{1,\{3\}}\sqcup \calM_{L_{1}, \{0,2\}}$$ called the Ekedahl-Oort stratification. The closure of $\calM^{\circ}_{L_{1}}$ is $\calM_{L_{1}}$ and $\calM^{\circ}_{L_{1}}$ is isomorphic to a unitary Deligne-Lusztig variety. The closure of $\calM^{\circ}_{L_{1}, \{3\}}$ is ismorphic to $\PP^{1}$ and the complement of $\calM^{\circ}_{L_{1}, \{3\}}$ in $\PP^{1}$ is precisely $\calM_{L_{1}, \{0, 2\}}$ which consists precisely the $\FF_{p^{2}}$-rational points of $\PP^{1}$. 

\item The closure of $\calM^{\circ}_{\{3\}}$ is $\calM_{\{3\}}=\bigcup_{L_{3}}\calM_{L_{3}}$ where  $L_{3}$ runs through all the vertex lattice of type $3$ and each $\calM_{L_{3}}$ is isomorphic to the Fermat hypersurface $x_{0}^{p+1}+x_{1}^{p+1}+x_{2}^{p+1}+x_{3}^{p+1}=0$. Moreover $\calM_{L_{3}}$ admits a stratification $$\calM_{L_{3}}=\calM^{\circ}_{L_{3}}\sqcup \calM^{\circ}_{3,\{1\}}\sqcup \calM_{L_{3}, \{0,2\}}$$ called the Ekedahl-Oort stratification. The closure of $\calM^{\circ}_{L_{3}}$ is $\calM_{L_{3}}$ and $\calM^{\circ}_{L_{3}}$ is isomorphic to a unitary Deligne-Lusztig variety. The closure of $\calM^{\circ}_{L_{3}, \{1\}}$ is ismorphic to $\PP^{1}$ and the complement of $\calM^{\circ}_{L_{3}, \{1\}}$ in $\PP^{1}$ is precisely $\calM_{L_{3}, \{0,2\}}$ which consists precisely the $\FF_{p^{2}}$-rational points of $\PP^{1}$.

\item The intersection between $\calM_{L_{1}}$ and $\calM_{L_{3}}$ for a vertex lattice $L_{1}$ of type $1$ and a vertex lattice $L_{3}$ of type 3, if nonempty, is isomorphic to a $\PP^{1}$. The intersection between $\calM_{L_{1}}$ and $\calM_{L^{\prime}_{1}}$ for two vertex lattices $L_{1}$ and $L^{\prime}_{1}$ of type $1$, if nonempty, is a point which is superspecial. The intersection between $\calM_{L_{3}}$ and $\calM_{L^{\prime}_{3}}$ for two vertex lattices $L_{3}$ and $L^{\prime}_{3}$ of type $3$, if nonempty, is  a point which is superspecial.

\end{enumerate}
\end{theorem}

\subsection{Application to the supersingular locus: $p$-inert case} The Rapoport-Zink uniformization theorem Theorem \eqref{RZ-uniform} bombined with the Theorem \ref{main-result-inert} proved above yield the following result about the supersingular locus $\Shim^{ss}_{U^{p}}$ in the case $p$ is inert in $E$. 
\begin{theorem}
The scheme $\Shim^{ss}_{U^{p}}$ is pure of dimension $2$. The irreducible components are all isomorphic to the Fermat hypersurface $x_{0}^{p+1}+x_{1}^{p+1}+x_{2}^{p+1}+x_{3}^{p+1}=0$. There are two types of irreducible components corresponding to the node $1$ and  node $3$ in the affine Dynkin diagram of type $\tilde{A}_{3}$. We refer to them as type $1$ components and type $3$ components. 
\begin{enumerate}
\item The intersection of a type $1$ component and a type $3$ component is a projective line if this intersection is non-empty;
\item The intersection of a type $1$ component with different type $1$ component is a supserspecial point if the intersection is non-empty;
\item The intersection of a type $3$ component with different type $3$ component is a supserspecial point if the intersection is non-empty.
\end{enumerate}
\end{theorem}

\section{Bruhat-Tits stratification of affine Deligne-Lusztig varieties}

In this final section we would like point out how our results fit in the results proved in G\"{o}rtz and He \cite{GH-Cam15} in terms of the affine Deligne-Lusztig varieties.
\subsection{Affine Deligne Lusztig variety} Let $F$ be a finite extension of $\QQ_{p}$ and $\breve{F}$ be the completion of the maximal unramified extension of $F$. Let $G$ be a connected reductive group over $F$ and we write $\breve{G}$ its base change to $\breve{F}$. Then $\breve{G}$ is quasi-split and we choose a maximal split torus $S$ and denote by $T$ its centralizer.  We know $T$ is a maximal torus and we denote by $N$ its normalizer. The relative Weyl group is defined to be $W=N(\breve{F})/ T(\breve{F})$. This is a finite group. Let $\Gamma$ be the Galois group of $\breve{F}$ and we have the following Kottwitz homomorphism \cite{RR96}:
$$\kappa_{G}: G(\breve{F})\rightarrow X_{*}(\breve{G})_{\Gamma}.$$
Denote by $\widetilde{W}$ the Iwahori Weyl group of $\breve{G}$ which is by definition $\widetilde{W}=N(\breve{F})/ T(\breve{F})_{1}$ where $T(\breve{F})_{1}$ is the kernel of the Kottwitz homomorphism for $T(\breve{F})$. Let $\breve{\mathfrak{B}}(G)$ be the Bruhat-Tits building of $G$ over $\breve{F}$. The choice of $S$ determines an standard apartment $\breve{\mathfrak{A}}$ which $\widetilde{W}$ acts on by affine transformations. We fix a $\sigma$-invariant alcove $\mathfrak{a}$ and a special vertex of $\mathfrak{a}$. Inside the Iwahori Weyl group $\widetilde{W}$, there is a copy of the affine Weyl group $W_{a}$ which can be identified with $N(\breve{F})\cap G(\breve{F})_{1}/ T(\breve{F})_{1}$ where $G(\breve{F})_{1}$ is the kernel of the Kottwitz morphism for $G(\breve{F})$. The group $\widetilde{W}$ is not quite a Coxeter group while $W_{a}$ is generated by the affine reflections denoted by $\tilde{\mathbb{S}}$ and $(\widetilde{W}, \tilde{\mathbb{S}})$ form a Coxeter system. We in fact have $\widetilde{W}= W_{a}\rtimes \Omega$ where $\Omega$ is the normalizer of a fixed base alcove $\mathfrak{a}$ and more canonically $\Omega=X_{*}(T)_{\Gamma}/ X_{*}(T_{sc})_{\Gamma}$ where $T_{sc}$ is the preimage of $T\cap G_{der}$ in the simply connected cover $G_{sc}$ of $G_{der}$.

Let $\mu\in X_{*}(T)$ be a minuscule cocharacter of $G$ over $\breve{F}$ and $\lambda$ its image in $X_{*}(T)_{\Gamma}$.  We denote by $\tau$ the projection of $\lambda$ in $\Omega$. The \emph{admissible subset} of $\widetilde{W}$ is defined to be
$$\Adm(\mu)=\{w\in \widetilde{W}; w \leq x(\lambda) \text{ for some }x\in W \}.$$
Here $\lambda$ is considered as a translation element in $\widetilde{W}$. Let $K\subset\tilde{\mathbb{S}}$ and $\breve{K}$ its corresponding parahoric subgroup. Let $\widetilde{W}_{K}$ be the subgroup defined by $N(\breve{F})\cap \breve{K}/T(\breve{F})_{1}$. We have the identification $\breve{K}\backslash G(\breve{F})/\breve{K}= \widetilde{W}_{K}\backslash \widetilde{W}/ \widetilde{W}_{K}$. Therefore we can define a relative position map 
\begin{equation}
\begin{split}
\inv: &G(\breve{F})/\breve{K}\times G(\breve{F})/\breve{K}\rightarrow \widetilde{W}_{K}\backslash \widetilde{W}/ \widetilde{W}_{K}\\
 &(g,h)\rightarrow \breve{K}g^{-1}h\breve{K}.\\
\end{split}
\end{equation}

Let $b\in G(\breve{F})$ be an element whose image in $B(G)$, the $\sigma$-conjugacy class of $G(\breve{F})$, lies in the subsets $B(G, \mu)$ of \emph{neutrally acceptable elements} see \cite[4.5, 4.6]{Rap-Ast05}. For $w\in \widetilde{W}_{K}\backslash \widetilde{W}/ \widetilde{W}_{K}$ and $b\in G(\breve{F})$, we define the \emph{affine Deligne-Lusztig variety}
to be the set 
$$X_{w}(b)=\{g\in G(\breve{F})/\breve{K}; \inv(g, b\sigma(g))=w\}.$$
Thanks to the work of \cite{BS-Inv17} and \cite{Zhu-Ann17}, this set can be viewed as an ind-closed-subscheme in the affine flag variety $\breve{G}/\breve{K}$. In this note, we only consider it as a set. The Rapoport-Zink space is not directly related to the affine Deligne-Lusztig variety but rather to the following union of affine Deligne-Lusztig varieties
$$X(\mu, b)_{K}=\{g\in G(\breve{F})/ \breve{K}; g^{-1}b\sigma(g)\in \breve{K}w\breve{K}, w\in \Adm(\mu) \}.$$
We recall the group $J_{b}$ is defined by the $\sigma$-centralizer of $b$ that is $$J_{b}(R)=\{g\in G(R\otimes_{F}\breve{F}); g^{-1}b\sigma(g)=b\}$$ for any $F$-algebra $R$. In the following we will assume that $b$ is \emph{basic} and in this case $J_{b}$ is an inner form of $G$ see \cite{RR96}.

\subsection{Coxeter type ADLV}We define $\Adm^{K}(\mu)$ to be the image of  $\Adm(\mu)$ in $\widetilde{W}_{K}\backslash\widetilde{W}/\widetilde{W}_{K}$ and $^{K}\widetilde{W}$ to be the set of elements of minimal length in $\widetilde{W}_{K}\backslash\widetilde{W}$.  We define the set $\mathrm{EO}^{K}(\mu)=\Adm^{K}(\mu)\cap ^{K}\widetilde{W}$. For $w\in W_{a}$, we set 
$$\mathrm{supp}_{\sigma}(w\tau)=\bigcup_{n\in \ZZ}(\tau\sigma)^{n}(\mathrm{supp}(w)).$$
If the length $l(w)$ of $w$ agrees with the cardinality of $\mathrm{supp}_{\sigma}(w\tau)/\langle\tau\sigma\rangle$, we say $w\tau$ is a $\sigma$-Coxeter element. We denote by $\mathrm{EO}^{K}_{\sigma,\mathrm{cox}}(\mu)$ the subset of $\mathrm{EO}^{K}(\mu)$ such that $w$ is a $\sigma$-Coxeter element and $\mathrm{supp}_{\sigma}(w)$ is not $\widetilde{\mathbb{S}}$. A \emph{$K$-stable piece} is a subset of $G(\breve{F})$ of the form $\breve{K}\cdot_{\sigma}\breve{I}w\breve{I}$ where $\cdot_{\sigma}$ means $\sigma$-conjugation and $\breve{I}$ is an Iwahori subgroup and $w\in {^{K}\widetilde{W}}$. Then we define the Ekedahl-Oort stratum attached to $w\in \mathrm{EO}^{K}(\mu)$ of $X(\mu, b)_{K}$ by the set $$X_{K,w}(b)=\{g\in G(\breve{F})/\breve{K}; g^{-1}b\sigma(g)\in  \breve{K}\cdot_{\sigma}IwI\}.$$  Then by \cite{GH-Cam15} we have the following Ekedahl-Oort stratification 
\begin{equation}X(\mu, b)_{K}=\bigcup_{w\in\mathrm{EO}^{K}(\mu)}X_{K,w}(b).\end{equation} 
The case when 
\begin{equation}\label{ADLV-EO}X(\mu, b)_{K}=\bigcup_{w\in\mathrm{EO}^{K}_{\sigma,\mathrm{cox}}(\mu)}X_{K,w}(b)\end{equation}
is particular interesting and when this happens we say the datum $(G, \mu, K)$ is of Coxeter type. The datum $(G, \mu, K)$ being Coxeter type or not depends only on the associated datum $(\widetilde{W}, \lambda, K, \sigma)$ where $\lambda$ is the image of $\mu\in X_{*}(T)_{I}$ and $\sigma$ is the induced automorphism of the Frobenius $\sigma$ on $\widetilde{W}$. The set of $(G, \mu, K)$ is classified in \cite{GH-Cam15} Theorem 5.11. This includes the two cases we studied in the previous sections.
\begin{itemize}
\item[-] The $p$-split case  corresponds to the datum $$(\tilde{A}_{3}, \omega^{\vee}_{2}, {\mathbb{S}}, id)$$ where $\sigma$ acts on the affine Dynkin diagram by the identity.
\item[-] The $p$-inert case corresponds to the datum $$(\tilde{A}_{3}, \omega^{\vee}_{2}, {\mathbb{S}}, \sigma_{0})$$
where $\sigma$ acts on the affine Dynkin diagram by fixing the nodes $0$ and $2$ and interchanging the nodes $1$ and $3$.
\end{itemize}

\subsection{Bruhat-Tits stratification of ADLV} Now we assume that $K$ is a maximal proper subset of $\tilde{\mathbb{S}}$ such that $\sigma(K)=K$. Consider the following set
$$\mathcal{J}=\{\Sigma\subset\tilde{\mathbb{S}}; \emptyset\neq \Sigma\text{ is }  \tau\sigma\text{-stable}\text{ and }d(v)=d(v^{\prime})\text{ for every } v,v^{\prime}\in \Sigma\}.$$ where $d(v)$ is the distance between $v$ and the unique vertex not in $K$.
In fact every $w\in \mathrm{EO}^{K}_{\sigma,\mathrm{cox}}(\mu)$ corresponds to a $\Sigma\in \mathcal{J}$ and we write $w$ as $w_{\Sigma}$.  If $(G,\mu, K)$ is of Coxeter type, for any $w_{\Sigma}\in\mathrm{EO}^{K}_{\sigma,\mathrm{cox}}(\mu)$, 
\begin{equation}
\label{EO-DL}X_{K,w_{\Sigma}}(b)=\bigcup_{i\in J_{b}/J_{b}\cap \breve{K}_{\tilde{\mathbb{S}}-\Sigma}}i X(w_{\Sigma}).
\end{equation} 
Here $\breve{K}_{\tilde{\mathbb{S}}-\Sigma}$ is the parahoric subgroup associated to the set $\tilde{\mathbb{S}}-\Sigma$ and $X(w_{\Sigma})$ is a classical Deligne-Lusztig variety defined by $$X(w_{\Sigma})=\{g\in \breve{K}_{\mathrm{supp}_{\sigma}(w_{\Sigma})}/\breve{I}; g^{-1}\tau\sigma(g)\in \breve{I}w\breve{I}\}$$ which is a Deligne-Lusztig variety attached to the maximal reductive quotient $G_{w}$ of the special fiber of $\breve{K}_{\tilde{\mathbb{S}}-\Sigma}$. Combine \eqref{ADLV-EO} and \eqref{EO-DL} we arrive at the following \emph{Bruhat-Tits stratification} of $X(\mu, b)_{K}$:
\begin{equation}
X(\mu, b)_{K}=\bigcup_{J_{b}/J_{b}\cap \ker(\kappa_{G})}\bigcup_{w_{\Sigma}\in  \mathrm{EO}^{K}_{\sigma,\mathrm{cox}}} \mathcal{X}^{\circ}_{\Sigma}
\end{equation}
where 
\begin{equation}\label{MSigma}
\mathcal{X}^{\circ}_{\Sigma}=\bigcup_{i\in J_{b}\cap\ker(\kappa_{G})/J_{b}\cap \breve{K}_{\tilde{\mathbb{S}}-\Sigma}}iX(w_{\Sigma}).
\end{equation}
Here the index set is related to the Bruhat-Tits building of $J_{b}$ in the following way. The group  $J_{b}\cap\ker(\kappa_{G})$ acts on the set of faces of type $\Sigma$ transitively and  $J_{b}\cap \breve{K}_{\tilde{\mathbb{S}}-\Sigma}$ is precisely the stabilizer of the face of type $\Sigma$ in the base alcove.

\subsubsection{$p$-split case} In the $p$-split case, we can compute
\begin{center}
\begin{tabular}{lllll}
$\Sigma$                                     & \{0,2\}                    & \{1,3\}   \\
$w_{\Sigma}  $                               & $\tau $     & $s_{0}\tau$   \\
$\tilde{\mathbb{S}}-\Sigma$             & \{1,3\}                  & \{0,2\}   \\
$\mathrm{supp}_{\sigma}(w_{\Sigma})$                             & $\emptyset $           & \{0,2\}.                                                
\end{tabular}
\end{center}

In this case the Deligne-Lusztig varieties $X(w_{\{0,2\}})$ is $0$-dimensional and $X(w_{\{1,3\}})$ is isomorphic the complement of $\FF_{p^{2}}$-points in $\PP^{1}(\FF)$. Therefore we have the following comparison between the Bruhat-Tits stratification for ADLV and Bruhat-Tits stratification for the Rapoport-Zink space studied in the $p$-split case. Recall that for $\calM$ the Bruhat-Tits stratification in Theorem \ref{main-p-split} is given by 
\begin{equation}\label{BT-stra-AFDL2}
\calM=\calM^{\circ }_{\{1,3\}}\sqcup \calM_{\{0,2\}}.
\end{equation}
\begin{itemize}
\item[-] $\mathcal{X}^{\circ}_{\{1,3\}}$ in \eqref{MSigma} can be identified with $\calM^{\circ}_{\{1,3\}}$ whose irreducible components are $X(w_{\{1,3\}})$;
\item[-] $\mathcal{X}_{\{0,2\}}$ in \eqref{MSigma} can be identified with $\calM_{\{0,2\}}$ and is the set of superspecial points.
\end{itemize}

\subsubsection{$p$-inert case} In the case $p$ is inert, we can compute
\begin{center}
\begin{tabular}{lllll}
$\Sigma$                                     & \{0,2\}                    & \{1,3\}                     & \{3\}                               & \{1\} \\
$w_{\Sigma}$                               & $\tau $     & $s_{0}\tau$ & $s_{0}s_{1}\tau$ &        $s_{0}s_{3}\tau$ \\
$\tilde{\mathbb{S}}-\Sigma$             & \{1,3\}                  & \{0,2\}                       & \{0,1,2\}                             & \{0,2,3\}\\
$\mathrm{supp}_{\sigma}(w_{\Sigma})$ & $\emptyset $& \{0,2\}                       & \{0,1,2\}                             & \{0,2,3\}.                            
\end{tabular}
\end{center}

In this case the Deligne-Lusztig varieties $X(w_{\{1\}})$ and $X(w_{\{3\}})$ agrees with $X_{B}(w_{2})$ in Theorem \ref{DL-stratification-inert}. The Deligne-Lusztig variety $X(w_{\{1,3\}})$ is isomorphic to $X_{B}(w_{1})$ and $X(w_{\{0,2\}})$ is $0$-dimensional and agrees with $X_{P_{1,3}}(1)$. Therefore we have the following comparison between the Bruhat-Tits stratification for ADLV and Bruhat-Tits stratification for the Rapoport-Zink space studied in the $p$-inert case. First recall that for $\calM$ the Brhat-Tits stratification in Theorem \ref{main-result-inert} is given by \begin{equation}\label{BT-stra-AFDL1}\calM=\calM^{\circ }_{\{1\}}\sqcup\calM^{\circ}_{\{3\}}\sqcup \calM^{\circ}_{\{1,3\}}\sqcup \calM_{\{0,2\}}.\end{equation}
\begin{itemize}
\item[-] $\mathcal{X}^{\circ}_{\{1\}}$ in \eqref{MSigma} is identified with $\calM^{\circ }_{\{1\}}$ in \eqref{BT-stra-AFDL1} whose irreducible components are $X_{B}(w_{2})$;
\item[-] $\mathcal{X}^{\circ}_{\{3\}}$ in \eqref{MSigma} is identified with $\calM^{\circ }_{\{3\}}$ in \eqref{BT-stra-AFDL1} whose irreducible components are $X_{B}(w_{2})$;
\item[-] $\mathcal{X}^{\circ}_{\{1, 3\}}$ in \eqref{MSigma}  is identified with $\calM^{\circ }_{\{1,3\}}$ in \eqref{BT-stra-AFDL1} whose irreducible components are $X_{B}(w_{1})$;
\item[-] $\mathcal{X}^{\circ}_{\{0,2\}}$ in \eqref{MSigma}  is identified with $\calM_{\{0,2\}}$ in \eqref{BT-stra-AFDL1} and is the set of superspecial points.
\end{itemize}

\end{document}